\documentclass[round,authoryear,11pt]{elsarticle}

\usepackage{soul}
\usepackage{amssymb}
\usepackage{lscape}
 \usepackage{amsthm}
 \usepackage{amsmath}
 \usepackage{natbib}
 \usepackage{tikz}
 \usepackage{booktabs}
 \usepackage{mathdots}
 \usepackage{lineno}
 \theoremstyle{plain}
 \newdefinition{definition}{Definition}
\newtheorem{theorem}{Theorem}
 \newtheorem{lemma}[theorem]{Lemma}
  \newtheorem{corollary}[theorem]{Corollary}
  \newtheorem{proposition}{Proposition}
 \newdefinition{remark}{Remark}
 \newdefinition{property}{Property}
  \newdefinition{example}{Example}
\usetikzlibrary{patterns}
\usepackage{bm}

\usepackage{mathtools}
\usepackage[pagebackref]{hyperref}
\usepackage{backref}

\DeclarePairedDelimiter\floor{\lfloor}{\rfloor}

\definecolor{mycolor2}{rgb}{0.12, 0.3, 0.17}
 \definecolor{mycolor}{rgb}{0,0,0}

\newcommand{\R}{\mathbb{R}}
\newcommand{\Z}{\mathbb{Z}}
\newcommand{\B}{\mathcal{B}}

\newcommand{\TUDC}{\mathcal{UDC}}
\newcommand{\TCDQ}{\mathcal{CDQ}}
\newcommand{\V}{\mathcal{V}}
\newcommand{\ASM}{\mathcal{AS}\!\mathcal{M}}

\def\newop#1{\expandafter\def\csname #1\endcsname{\mathop{\rm
#1}\nolimits}}
\newcommand{\block}[2]{
  \underbrace{#1 \cdots #1}_{#2}}

\newop{conv}
\newop{UDC}
\newop{CDQ}
\newop{DC}
\newop{DQ}
\newop{D}
\newop{ID}

\journal{arXiv}

\begin{document}

\begin{frontmatter}

\title{Geometry of Discrete Copulas}

\author[MIT]{Elisa Perrone}
\ead{eperrone@mit.edu}
\author[KTH]{Liam Solus}
\ead{solus@kth.se}
\author[MIT]{Caroline Uhler}
\ead{cuhler@mit.edu}

\address[MIT]{Laboratory for Information and Decision Systems, and Institute for Data, Systems and Society, Massachusetts Institute of Technology, Cambridge, MA, USA}
\address[KTH]{KTH Royal Institute of Technology, Stockholm, Sweden}

\begin{abstract}
Multivariate distributions are fundamental to modeling. Discrete copulas can be used to construct diverse multivariate joint distributions over random variables from estimated univariate marginals.  
The space of discrete copulas admits a representation as a convex polytope which can be exploited in entropy-copula methods relevant to hydrology and climatology. 
To allow for an extensive use of such methods in a wide range of applied fields, it is important to have a geometric representation of discrete copulas with desirable stochastic properties.
In this paper, we show that the families of ultramodular discrete copulas and their generalization to convex discrete quasi-copulas admit representations as polytopes.  
We draw connections to the prominent Birkhoff polytope, alternating sign matrix polytope, and their most extensive generalizations in the discrete geometry literature.  
In doing so, we generalize some well-known results on these polytopes from both the statistics literature and the discrete geometry literature.
\end{abstract}

%%%%%%KEYWORD%%%%%%%
\begin{keyword}
Discrete (quasi-)copulas \sep ultramodularity \sep transportation polytope \sep alternating sign matrix polytope \sep Birkhoff polytope.
%% keywords here, in the form: keyword \sep keyword
%% PACS codes here, in the form: \PACS code \sep code
%% MSC codes here, in the form: \MSC code \sep code
%% or \MSC[2008] code \sep code (2000 is the default)
\end{keyword}

\end{frontmatter}

%% \linenumbers
%---SECTION: Introduction-----
\section{Introduction}
\label{sec: introduction}
Multivariate probability distributions with ordinal or interval support are fundamental to a wide range of applications, including health care, weather forecasting, and image analysis.  
While it is straightforward to estimate univariate marginal distributions, it is often challenging to model multivariate and high-dimensional joint distributions.  
Copulas \citep{DurSem15,nelsen_06} serve as a general toolbox for constructing multivariate distributions from the estimated univariate marginals and can be equipped with different stochastic dependence properties such as exchangeability, positive/negative association, or tail dependence.
Key to the power of copulas is \emph{Sklar's Theorem}, which states that the joint distribution function $F_{\mathbf{X}}$ of \emph{any} $d$-dimensional random vector $\mathbf{X}=(X_1, \ldots,X_d)\in\mathbb{R}^d$ with univariate margins $F_{X_1}, \ldots, F_{X_d}$, can be expressed as
\begin{equation*}
F_{\mathbf{X}}(x_1,\ldots, x_d) = C(F_{X_1}(x_1),\ldots,F_{X_d}(x_d)),
\end{equation*}
where the function $C: \textrm{Range}(F_{X_1}) \times \cdots \times \textrm{Range}(F_{X_d}) \longrightarrow \mathbb{R}$ is uniquely defined and known as \emph{$d$-dimensional copula}~\citep{sklar_59}. 
In the case of purely discrete random vectors, Sklar's theorem identifies the \emph{discrete copulas}.

In view of their probabilistic meaning, it is beneficial to have a wealth of copula functions with application-specific properties.
As directly constructing copulas with desirable features is a challenging task, researchers often focus on identifying stochastic properties that may serve as a tool for copula constructions.
A property known as \emph{ultramodularity} \citep{marinacci_2005} is particularly desirable while aiming at constructing new copulas \citep{klement_2011,klement_2014,klement_2017b,Saminger-Platz_2017}.
In this paper, we show that bivariate discrete copulas with properties such as ultramodularity admit \emph{polytopal representations} and thereby demonstrate that the analysis of stochastic dependence via copulas is amenable to techniques from \emph{convex geometry} and \emph{linear optimization}.

A \emph{(convex) polytope} is a bounded convex body in $\R^n$ that consists of the points $(x_1,\ldots,x_n)\in\R^n$ satisfying finitely many \emph{affine inequalities}
\begin{equation}
\label{eqn: h-rep definition}
a_1x_1+a_2x_2+\cdots+a_nx_n \leq b,
\end{equation}
where $a_1,\ldots,a_n,b\in\R$.  
A collection of such inequalities is called an \emph{$H$-representation} of the associated polytope.  
The unique irredundant $H$-repre\-sentation of a polytope $P$ is called its \emph{minimal $H$-representation}.  
If inequality~(\ref{eqn: h-rep definition}) is included in the minimal $H$-representation of $P$, then the points in $P$ on which (\ref{eqn: h-rep definition}) achieves equality is the associated \emph{facet} of $P$.  
Thus, the size of the minimal $H$-representation of $P$ is the number of facets of $P$.
Polytopes are fundamental objects in the field of \emph{linear optimization}, where a key goal is to decide if a polytope has a small minimal $H$-representation as to more efficiently solve associated linear programming problems.

Discrete copulas are known to admit a representation as a convex polytope \citep{kolesarova_2006}, and such representations 
have already been used to apply linear optimization techniques to solve copula-related problems in environmental sciences \citep{aghakouchak_2014,radi_2017}. 
Polytopal representations of discrete copulas are particularly useful when researchers desire copulas with \emph{maximum entropy}; i.e., those copulas that can be used to recover the least prescriptive distributions accounting for the limited data information available.
For instance, the geometric description of discrete copulas has been used in hydrology and climatology to derive copulas with maximum entropy that also match a known grade correlation coefficient \citep{piantadosi_2007,piantadosi_2012}.
Here, we extend the families of discrete copulas, and their generalizations, that have known polytopal representations to facilitate the further use of similar linear optimization techniques in the identification of copulas with maximum entropy and desirable stochastic properties such as ultramodularity. 
Our results allow obtaining simple ultramodular copulas that may serve as a tool for copula construction, and for dependence modeling in applied fields where ultramodularity is a desirable property, such as portfolio risk optimization \citep{mueller_2001} and risk aversion \citep{Mueller_2011}.

The space of discrete copulas in the bivariate setting (i.e., $d=2$) was studied by \cite{aguilo_2008,aguilo_2010,kolesarova_2006,Mayor_2005,Mesiar_2005,mordelova_2017}.  
The results of these papers collectively demonstrate that the space of bivariate discrete copulas constructed from marginal distributions with finite state spaces of sizes $p$ and $q$ correspond to the points within a special polytope known as the \emph{generalized Birkhoff polytope} \citep{Ziegler_95}.  
\citet{klement_2011} gave a functional characterization of ultramodular copulas.  
Here, we use their characterization to identify the minimal $H$-representation of ultramodular discrete copulas as a \emph{subpolytope} of the generalized Birkhoff polytope.  

Bivariate copulas admit an important generalization that in turn results in a natural generalization of the generalized Birkhoff polytope.  
The bivariate copulas form a \emph{poset} $P$ with partial order $\prec$ defined as $C\prec C^\prime$ whenever $C(u,v)< C^\prime(u,v)$ for all $(u,v)\in[0,1]^2$ \citep{DurSem15}.  
However, $P$ fails to admit desirable categorical properties.  
In particular, $P$ is not a \emph{lattice}, meaning that not all pairs of copulas, $C$ and $C^\prime$, have both a least upper bound and greatest lower bound with respect to $\prec$. 
The family of functions that complete $P$ to a lattice under $\prec$ are known as \emph{quasi-copulas} \citep{Nelsen_2005}, and in the case where $p = q$, the bivariate discrete quasi-copulas correspond to points within a polytope known as the \emph{alternating sign matrix polytope} \citep{striker_2009}.  
In this paper, we identify the minimal $H$-representations for the family of \emph{discrete quasi-copulas} with $p\neq q$ and the subfamily of \emph{discrete quasi-copulas with convex sections} residing within.  
Notably, we generalize a theorem of \citet[Theorem 3.3]{striker_2009}  by showing that the \emph{alternating transportation polytopes} \citep{Knight_2009} have minimal $H$-representations whose size is quadratic in $p$ and $q$, a result of independent interest in discrete geometry.

The remainder of this paper is organized as follows: 
In Section~\ref{sec: copulas and quasi-copulas in discrete geometry}, we provide basic definitions. 
In Section~\ref{sec: the polytope of ultramodular discrete copulas}, we present our first main result (Theorem~\ref{Th:h-repUC}), in which we show that the collection of ultramodular bivariate discrete copulas is representable as a polytope, and we identify its minimal $H$-representation. The statistical significance of this result is that it allows rephrasing the problem of selecting an ultramodular bivariate discrete copula as an efficient linear optimization problem.
Similarly, the collection of bivariate discrete quasi-copulas is also representable as a polytope generalizing the alternating sign matrix polytope.  
In Section~\ref{sec: the polytopes of discrete quasi-copulas and convex discrete quasi-copulas}, we give our second main result (Theorem~\ref{thm: facets of the generalized alternating sign matrix polytope}), in which we identify the minimal $H$-representation of this polytope, thereby generalizing a result in discrete geometry \citep{striker_2009}. In addition, we identify the  minimal $H$-representation of a subpolytope corresponding to the discrete quasi-copulas with convex sections.  
In Section~\ref{sec: on vertex representations}, we analyze alternative representations of these polytopes; i.e., their sets of \emph{vertices}.  
Finally, in Section~\ref{sec: aggregation functions and alternating transportation polytopes}, we show that the most extensive generalization of bivariate discrete copulas in the statistical literature admits a characterization in terms of the most extensive generalization of the Birkhoff polytope in the discrete geometry literature, thereby completely unifying these two hierarchies. 
Collectively, these results provide new and potentially useful geometric perspectives on important families of discrete copulas and quasi-copulas and introduce previously unstudied polytopes that may be of independent interest to researchers in discrete geometry.

%---SECTION: Copulas and Quasi-Copulas in Discrete Geometry----
\section{Copulas and Quasi-copulas in Discrete Geometry}
\label{sec: copulas and quasi-copulas in discrete geometry}
In this section, we present the statistical and geometric preliminaries we will use throughout the paper.  
We first recall definitions and fundamental results for copulas and quasi-copulas.  
We then explicitly define the polytopes we will study in the remaining sections.  
The following defines bivariate copulas by way of functional inequalities.
%---DEFINITION: Copulas-------
\begin{definition}
\label{Def:cop}
A function $C:[0,1]^2 \rightarrow [0,1]$ is a \emph{copula} if and only if
\begin{enumerate}
\item[(C1)] for every $u \in [0,1]$, $C(u,0)=C(0,u)=0$ and $C(u,1)=C(1,u)=u$;
	\vspace{3pt}
\item[(C2)] for every $u_1, u_2, v_1, v_2 \in [0,1]$ s.t. $u_1 \leq u_2, v_1 \leq v_2$, it holds that
\begin{equation}
\label{2-incr}
C(u_1,v_1) + C(u_2,v_2) \geq C(u_1,v_2) + C(u_2,v_1).
\end{equation}
\end{enumerate}
\end{definition}
Hence, bivariate copulas are functions on the unit square that are uniform on the boundary (C1), \emph{supermodular} (C2), and that capture the joint dependence of random vectors.
A \emph{(coordinatewise) section} of a bivariate copula is any function given by fixing one of the two variables. A copula is \emph{ultramodular} if and only if all of its coordinatewise sections are convex functions~\citep{klement_2011,klement_2014}.
The following generalizes bivariate copulas:
%---DEFINITION: Bivariate Discrete Quasi-copulas-----
\begin{definition}\citep{genest_characterization_1999}
\label{Def:quasi-copula}
A function $Q:[0,1]^2 \rightarrow [0,1]$ is a \emph{quasi-copula} if and only if it satisfies condition (C1) of Definition \ref{Def:cop},
\begin{enumerate}
\item[(Q2)] $Q$ is increasing in each component, and
	\vspace{3pt}
\item[(Q3)] $Q$ satisfies the 1-Lipschitz condition, i.e., $\forall u_1,u_2,v_1,v_2 \in [0,1],$ 
\[ |Q(u_2,v_2) -Q(u_1,v_1)| \leq |u_1 - v_1| + |u_2 - v_2|.\]
\end{enumerate}
\end{definition}
Equivalently, \citet{genest_characterization_1999} show that bivariate quasi-copulas are functions that satisfy the boundary condition (C1) and are supermodular on any rectangle with at least one edge on the boundary of the unit square.

%---SUBSECTION: Polytopes for Copulas and Quasi-copulas----
\subsection{Polytopes for Copulas and Quasi-copulas}
\label{subsec: polytopes for copulas and quasi-copulas}
In the following, for $p\in\Z_{>0}:= \mathbb{N} \setminus \{0\}$ we let $[p] := \{1,\ldots,p\}$, $\langle p\rangle := \{0,1,\ldots,p\}$, and $I_p $ {\footnotesize$:=\left\{ 0, \tfrac{1}{p}, \ldots, \tfrac{p-1}{p}, 1 \right\}$}.  
When the marginal state spaces of a discrete (quasi)-copula $C_{p,q}: I_p\times I_q \longrightarrow [0,1]$ are of sizes $p$ and $q$, respectively, we can then define it on the domain $I_p\times I_q$.  
It follows that $C_{p,q}$ is representable with a $(p+1)\times(q+1)$ matrix $C = [c_{ij}]$, where $c_{ij} := C_{p,q}(i/p,j/q)$. 
We can then define the set of discrete copulas on $I_p\times I_q$, denoted by $\DC_{p,q}$, to be all matrices $[c_{ij}]\in\R^{(p+1)\times(q+1)}$ satisfying the affine inequalities
\begin{enumerate}
	\item[(c1)] $c_{0j}=0,\; c_{pj}=\frac{j}{q}; \;\,  c_{i0}=0,\; c_{iq}=\frac{i}{p} \quad\,  \mbox{for all } \, i \in \langle p\rangle, j \in \langle q\rangle$;
	\vspace{3pt}
	\item[(c2)] $c_{ij} + c_{i-1, j-1} - c_{i,j-1} - c_{i-1,j} \geq 0 \quad \mbox{for all } \, i \in [p], j \in [q].$
		\vspace{3pt}
\end{enumerate}
Analogously, the polytope of discrete quasi-copulas on $I_p\times I_q$ is denoted by $\DQ_{p,q}$ and it consists of all matrices $[c_{ij}]\in\R^{(p+1)\times(q+1)}$ satisfying
\begin{enumerate}
	\item[(q1)] $c_{0j}=0,\; c_{pj}=\frac{j}{q}; \,\;  c_{i0}=0,\; c_{iq}=\frac{i}{p} \quad \mbox{for all } \, i \in \langle p\rangle, j \in \langle q\rangle;$
	\vspace{3pt}
	\item[(q2a)] $0 \leq c_{i+1, j} - c_{i j} \leq \frac{1}{p} \quad \mbox{for all } \, i \in \langle p-1\rangle, j \in [q];$
		\vspace{3pt}
	\item[(q2b)] $0 \leq c_{i,j+1} - c_{i j} \leq \frac{1}{q} \quad \mbox{for all } \, i \in [p], j \in \langle q-1\rangle.$
		\vspace{3pt}	
\end{enumerate}

We now recall the definitions of some classically studied polytopes in discrete geometry and show how they relate to the polytopes $\DC_{p,q}$ and $\DQ_{p,q}.$  
Given two vectors $u:=(u_1,\ldots,u_p)\in\R^p_{>0}$ and $v :=(v_1,\ldots,v_q)\in\R^q_{>0}$, the \emph{transportation polytope} $\mathcal{T}(u,v)$ is the convex polytope defined in the $pq$ variables $x_{ij}$ for $i\in[p]$ and $j\in[q]$ satisfying 
$$\mbox{$x_{ij}\geq0,$}
\qquad
\mbox{$\sum_{h=1}^qx_{ih} = u_i,$}
\qquad
\mbox{and}
\qquad
\mbox{$\sum_{\ell=1}^px_{\ell j} = v_j,$\vspace{-0.1cm}}
$$
for all $i\in[p]$ and $j\in[q]$.  
The vectors $u$ and $v$ are called the \emph{margins} of $\mathcal{T}(u,v)$.  
Transportation polytopes capture a number of classically studied polytopes in combinatorics \citep{DeLoera_14}.  
For example, the \emph{$p^{th}$ Birkhoff polytope}, denoted by $\B_p$, is the transportation polytope $\mathcal{T}(u,v)$ with $u = v = (1,1,\ldots,1)^T\in\R^p$, and the \emph{$p\times q$ generalized Birkhoff polytope}, denoted by $\B_{p,q}$, is the transportation polytope $\mathcal{T}(u,v)$ where $u = (q,q,\ldots,q)\in\R^p$ and $v = (p,p,\ldots,p)\in\R^q$.  

Another combinatorially-well-studied polytope that contains $\B_p$ is given by the convex hull of all \emph{alternating sign matrices}, i.e., square matrices with entries in \{0,1,-1\} such that the sum of each row and column is 1 and the nonzero entries in each row and column alternate in sign.  
\citet[Theorem 2.1]{striker_2009} proved that this polytope, known as the \emph{alternating sign matrix polytope} and denoted by $\ASM_p$, is defined by
\[
\mbox{$0\leq\sum_{\ell=1}^{i}x_{\ell j}\leq 1,$}
\quad
\mbox{$0\leq\sum_{h=1}^{j}x_{ih}\leq 1, \quad \sum_{i=1}^nx_{ij}=1,$} 
\quad
\mbox{$\sum_{j=1}^nx_{ij}=1,$}
\]
for all $i,\ell,j,h\in[n]$.   
Given margins $u\in\R^p$ and $v\in\R^q$, $\ASM_p$ was generalized to the \emph{alternating transportation polytope} $\mathcal{A}(u,v)$ \citep[Chapter 5]{Knight_2009}, consisting of all $p\times q$ matrices $[x_{ij}]\in\R^{p\times q}$ satisfying
\begin{enumerate}
	\item[(1)] $\sum_{\ell = 1}^p x_{\ell j} = v_j$;\; $\sum_{h = 1}^q x_{ih} = u_i$ \;for $i\in[p]$ and $j\in[q]$,
	\item[(2)] $0\leq\sum_{\ell = 1}^i x_{\ell j}\leq v_j$\; for all $i\in[p]$ and $j\in[q]$,
	\item[(3)] $0\leq \sum_{h = 1}^j x_{ih} \leq u_i$\; for all $i\in[p]$ and $j\in[q]$.
\end{enumerate}
Analogous to the generalized Birkhoff polytope, we define the \emph{generalized alternating sign matrix polytope}, denoted $\ASM_{p,q}$, to be the alternating transportation polytope $\mathcal{A}(u,v)$ with $u = (q,q,\ldots,q)^T\in\R^p$ and $v = (p,p,\ldots,p)^T\in\R^q$.
As shown in Proposition~\ref{prop: unimodular equivalence}, there is an (invertible) linear transformation taking each discrete copula $[c_{ij}]\in\R^{(p+1)\times(q+1)}$ to a matrix $[b_{ij}]\in\B_{p,q}$ and taking each discrete quasi-copula to a matrix in $\ASM_{p,q}$.  
In the following result we show that this linear transformation, which is well-known in the statistical literature, is also geometrically nice.

%---PROPOSITION: Unimodularly Equivalent Polytopes----
\begin{proposition}
\label{prop: unimodular equivalence}
The polytopes $\DC_{p,q}$ and $\frac{1}{pq}\B_{p,q}$ are unimodularly equivalent, as are the polytopes $\DQ_{p,q}$ and $\frac{1}{pq}\ASM_{p,q}$.  
\end{proposition}

\begin{proof}
Recall that two polytopes $\mathcal{P}$ and $\mathcal{Q}$ are unimodularly equivalent if and only if there exists a unimodular transformation $L$ from $\mathcal{P}$ to $\mathcal{Q}$, i.e., $L: \mathcal{P}\rightarrow \mathcal{Q}, \, x \mapsto \mathbf{A}x'$ is a linear transformation such that $\det(\mathbf{A})=\pm 1$.
It can be seen that there is a linear map $T:\R^{(p+1)\times(q+1)}\longrightarrow\R^{p\times q}$ for which $T(c_{ij}) := c_{ij}+ c_{i-1, j-1} - c_{i,j-1} - c_{i-1,j}$ for all $i\in[p]$ and $j\in[q]$ that takes a discrete copula to a matrix in $\frac{1}{pq}\B_{p,q}$.  
Similarly, the linear map $T$ takes a discrete quasi-copula to a matrix in $\frac{1}{pq}\mathcal{A}(u,v)$.  
Using the boundary condition (c1), the map $T$ is then an invertible transformation on $\R^{p\times q}$, and if we let $e_{ij}$ denote the standard basis vectors for $\R^{p\times q}$ ordered lexicographically (i.e. $e_{ij}< e_{kr}$ if and only if $i<k$ or $i=k$ and $j<r$), then we see that the matrix for the map $T$ is lower triangular and has only ones on the diagonal when the standard basis is chosen with the lexicographic ordering on the columns and rows.  
Therefore, $T$ is unimodular.
\end{proof}

%---REMARK: Geometrically Nice----
\begin{remark}
\label{rmk: geometrically nice}
Proposition~\ref{prop: unimodular equivalence} shows that the geometry of $\B_{p,q}$ and $\ASM_{p,q}$ completely describes the geometry of the collection of discrete copulas and discrete quasi-copulas, respectively. In particular, $\DC_{p,q}$ and $\B_{p,q}$ have the same facial structure, and similarly for $\DQ_{p,q}$ and $\ASM_{p,q}$.  
In addition, for any subpolytopes $P\subset\DC_{p,q}$ and $Q\subset\DQ_{p,q}$ the subpolytopes $T(P)\subset\B_{p,q}$ and $T(Q)\subset\ASM_{p,q}$ have the same facial structure, respectively. \qed
\end{remark}

The \emph{polytope of ultramodular discrete copulas} is the subpolytope $\UDC_{p,q}\subset\DC_{p,q}$ satisfying the additional constraints
\begin{equation}
\label{eqn: convexity constraints}
2c_{ij}\leq c_{i-1,j}+c_{i+1,j}
\qquad
\mbox{and}
\qquad
2c_{ij}\leq c_{i,j-1}+c_{i,j+1},
\end{equation}
for all $i\in[p-1]$ and $j\in[q-1]$.
These constraints correspond to convexity conditions imposed on the associated copulas, and so we can naturally define a similar subpolytope of $\DQ_{p,q}$.  
The \emph{polytope of convex discrete quasi-copulas} is the subpolytope $\CDQ_{p,q}\subset\DQ_{p,q}$ satisfying the above constraints~(\ref{eqn: convexity constraints}).  
Via the transformation $T$, we will equivalently study the polytopes $\TUDC_{p,q} := pqT(\UDC_{p,q})\subset\B_{p,q}$ and $\TCDQ_{p,q} := pqT(\CDQ_{p,q})\subset\ASM_{p,q}$.  
We end this section with a second geometric remark.
%---REMARK: DIMENSIONS-----
\begin{remark}
\label{rem: dimensions}
It is well known that the generalized Birkhoff polytope $\B_{p,q}$ has dimension $(p-1)(q-1)$ (see \citep{DeLoera_14} for instance).  
This is because each of the defining equalities $\sum_{i = 1}^p x_{ij} =p$ and $\sum_{j = 1}^q x_{ij} = q$ determine precisely one more entry of the matrix.  
In a similar fashion, the polytopes $\TUDC_{p,q}$, $\ASM_{p,q}$, and $\TCDQ_{p,q}$ and also 
the polytopes of discrete (quasi)-copulas $\DC_{p,q}, \UDC_{p,q}, \DQ_{p,q}$, and $\CDQ_{p,q}$ studied in this paper all have dimension $(p-1)(q-1)$.  \qed
\end{remark}

%---SECTION: The Polytope of Ultramodular Discrete Copulas----
\section{The Polytope of Ultramodular Discrete Copulas}
\label{sec: the polytope of ultramodular discrete copulas}
In our first main theorem we identify the minimal $H$-represesentation of the polytope of ultramodular discrete copulas $\UDC_{p,q}$.  
%---THEOREM: Minimal H-representation of the Polytope of Ultramodular Discrete Copulas----
\begin{theorem}
\label{Th:h-repUC}
The minimal $H$-representation of the polytope of ultramodular discrete copulas $\UDC_{p,q}$ consists of the 
$
(p-2)(q-2) + 2 (p-1)(q-1)
$
inequalities:
\begin{enumerate}
	\item[(d1)] $x_{11}\geq 0, \mbox{ and }\; x_{p-1,q-1} \geq \tfrac{(p-1)(q-1)-1}{pq},$
	\vspace{3pt}
	\item[(d2)] $x_{i j} + x_{i+1, j+1} - x_{i,j+1} - x_{i+1, j} \geq 0 \quad \mbox{for all } i \in [p-2], j \in [q-2] \text{ with }$
	
	 $(i,j) \notin \{(1,1),(p-2,q-2)\},$
		\vspace{3pt}
	\item[(d3a)] $x_{i j} + x_{i,j+2} - 2\, x_{i,j+1} \geq 0  \quad \mbox{for all } i \in [p-1], \; j \in \langle q-2\rangle,$
		\vspace{3pt}
	\item[(d3b)] $x_{i j} + x_{i+2,j} - 2\, x_{i+1, j} \geq 0  \quad \mbox{for all } j \in [q-1], \; i \in \langle p-2\rangle.$
\end{enumerate}
\end{theorem}
%---FIGURE: Representation of Inequalities----------
\begin{figure}[t!]
\begin{tikzpicture}[scale=0.4]
\draw (1,1) rectangle (11,11);
\draw (11,1) -- (11,11);
\draw (1,1) -- (11,1);
\draw (0.5,0.3)   node {{\scriptsize $c_{50}$}};
\draw (12,11.5)   node {{\scriptsize$c_{07}$}};
\draw (2.25,11.5)   node {{\scriptsize$c_{01}$}};
\draw (3.75,11.5)   node {{\scriptsize$c_{02}$}};
\draw (5.25,11.5)   node {{\scriptsize$c_{03}$}};
\draw (6.75,11.5)   node {{\scriptsize$c_{04}$}};
\draw (8.25,11.5)   node {{\scriptsize$c_{05}$}};
\draw (9.75,11.5)   node {{\scriptsize$c_{06}$}};
\draw (11.5,0.3)   node {{\scriptsize$c_{57}$}};
\draw (0.3,11.5)   node {{\scriptsize$c_{00}$}};
\draw (2.25,0.3)   node {{\scriptsize$c_{51}$}};
\draw (3.75,0.3)   node {{\scriptsize$c_{52}$}};
\draw (5.25,0.3)   node {{\scriptsize$c_{53}$}};
\draw (6.75,0.3)   node {{\scriptsize$c_{54}$}};
\draw (8.25,0.3)   node {{\scriptsize$c_{55}$}};
\draw (9.75,0.3)   node {{\scriptsize$c_{56}$}};
\draw (12.0,9)   node {{\scriptsize$c_{17}$}};
\draw (12.0,7)   node {{\scriptsize$c_{27}$}};
\draw (12.0,5)   node {{\scriptsize$c_{37}$}};
\draw (12.0,3)   node {{\scriptsize$c_{47}$}};
\draw (0.3,9)   node {{\scriptsize$c_{10}$}};
\draw (0.3,7)   node {{\scriptsize$c_{20}$}};
\draw (0.3,5)   node {{\scriptsize$c_{30}$}};
\draw (0.3,3)   node {{\scriptsize$c_{40}$}};

\draw[thick, pattern=north east lines, pattern color=red!30] (1,9) rectangle (2.43,11);
\draw[ thick, pattern=north east lines, pattern color=red!30] (9.58,1) rectangle (11,3);
\draw[ thick,fill=red!10] (2.43,5) rectangle (3.86,7);
\draw[ thick, fill=red!10] (3.86,7) rectangle (5.29,9);
\draw[ thick, fill=red!10] (3.86,5) rectangle (5.29,7);
\draw[ thick, fill=red!10] (2.43,3) rectangle (3.86,5);
\draw[ thick, fill=red!10] (5.29,5) rectangle (6.72,7);
\draw[ thick, fill=red!10] (5.29,7) rectangle (6.72,9);
\draw[ thick, fill=red!10] (6.72,5) rectangle (8.15,7);
\draw[ thick, fill=red!10] (6.72,3) rectangle (8.15,5);
\draw[ thick, fill=red!10] (8.15,5) rectangle (9.58,7);
\draw[ thick, fill=red!10] (8.15,7) rectangle (9.58,9);
\draw[ thick, fill=red!10] (5.29,3) rectangle (6.72,5);
\draw[ thick, fill=red!10] (6.72,7) rectangle (8.15,9);
\draw[ thick, fill=red!10] (3.86,3) rectangle (5.29,5);
\draw (10.35,2)   node {{\tiny d1}};
\draw (1.75,10)   node {{\tiny d1}};
\draw (3.15,6)   node {{\tiny d2}};
\draw (6.0,6)   node {{\tiny d2}};
\draw (7.45,6)   node {{\tiny d2}};
\draw (8.9,6)   node {{\tiny d2}};
\draw (4.55,6)   node {{\tiny d2}};
\draw (6.05,8)   node {{\tiny d2}};
\draw (7.45,8)   node {{\tiny d2}};
\draw (8.9,8)   node {{\tiny d2}};
\draw (4.55,8)   node {{\tiny d2}};
\draw (3.15,4)   node {{\tiny d2}};
\draw (6.05,4)   node {{\tiny d2}};
\draw (7.5,4)   node {{\tiny d2}};
\draw (4.5,4)   node {{\tiny d2}};

\draw[dotted] (2.43,1) -- (2.43,11);
\draw[dotted] (3.86,1) -- (3.86,11);
\draw[dotted] (5.29,1) -- (5.29,11);
\draw[dotted] (6.72,1) -- (6.72,11);
\draw[dotted] (8.15,1) -- (8.15,11);
\draw[dotted] (9.58,1) -- (9.58,11);
\draw[dotted] (1,3) -- (11,3);
\draw[dotted] (1,5) -- (11,5);
\draw[dotted] (1,7) -- (11,7);
\draw[dotted] (1,9) -- (11,9);

\draw[black, very thick] (2.33,2.9) rectangle (2.53,3.1);
\draw[black, very thick] (2.33,4.9) rectangle (2.53,5.1);
\draw[black, very thick] (2.33,6.9) rectangle (2.53,7.1);
\draw[black, very thick] (2.33,8.9) rectangle (2.53,9.1);

\draw[black, very thick] (3.76,2.9) rectangle (3.96,3.1);
\draw[black, very thick] (3.76,4.9) rectangle (3.96,5.1);
\draw[black, very thick] (3.76,6.9) rectangle (3.96,7.1);
\draw[black, very thick] (3.76,8.9) rectangle (3.96,9.1);

\draw[black, very thick] (5.19,2.9) rectangle (5.39,3.1);
\draw[black, very thick] (5.19,4.9) rectangle (5.39,5.1);
\draw[black, very thick] (5.19,6.9) rectangle (5.39,7.1);
\draw[black, very thick] (5.19,8.9) rectangle (5.39,9.1);

\draw[black, very thick] (6.62,2.9) rectangle (6.82,3.1);
\draw[black, very thick] (6.62,4.9) rectangle (6.82,5.1);
\draw[black, very thick] (6.62,6.9) rectangle (6.82,7.1);
\draw[black, very thick] (6.62,8.9) rectangle (6.82,9.1);

\draw[black, very thick] (8.05,2.9) rectangle (8.25,3.1);
\draw[black, very thick] (8.05,4.9) rectangle (8.25,5.1);
\draw[black, very thick] (8.05,6.9) rectangle (8.25,7.1);
\draw[black, very thick] (8.05,8.9) rectangle (8.25,9.1);

\draw[black, very thick] (9.48,2.9) rectangle (9.68,3.1);
\draw[black, very thick] (9.48,4.9) rectangle (9.68,5.1);
\draw[black, very thick] (9.48,6.9) rectangle (9.68,7.1);
\draw[black, very thick] (9.48,8.9) rectangle (9.68,9.1);
\end{tikzpicture}
\centering
\caption{A depiction of the inequalities for the minimal $H$-representation of $\UDC_{5,7}$.  The rectangles represent the necessary supermodularity constraints (C2), while the square dots represent the convexity contraints~(\ref{eqn: convexity constraints}).}
\label{Fig:UDC}
\end{figure}
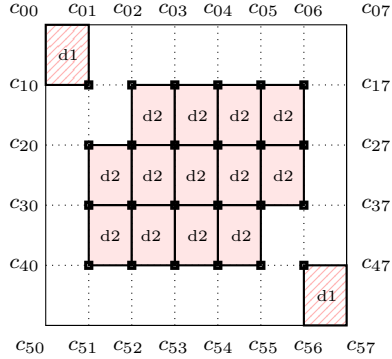
Figure \ref{Fig:UDC} gives a diagrammatic depiction of the inequalities constituting the minimal $H$-representation of $\UDC_{p,q}$. 
An equivalent statement to Theorem~\ref{Th:h-repUC} is that the subpolytope $\TUDC_{p,q}$ of the generalized Birkhoff polytope $\B_{p,q}$ has minimal $H$-representation given by the inequalities
\begin{enumerate}
	\item[(b1)] $x_{11}\geq 0, \mbox{ and }\; x_{p q} \geq 0,$
	\vspace{3pt}
	\item[(b2)] $x_{i+1, j+1}\geq 0 \; \mbox{for all } i \in [p-2], \, j \in [q-2] \text{ with }$
	\vspace{3pt}
	\item[] $(i,j) \notin \{(1,1),(p-2,q-2)\},$
	\vspace{3pt}
	\item[(b3a)] $\sum\limits_{\ell=1}^i x_{\ell,j+1} \geq  \sum\limits_{\ell=1}^i x_{\ell j} \, \mbox{for all } i \in [p-1], \, j \in [q-1],$
	\vspace{3pt}
	\item[(b3b)] $\sum\limits_{h=1}^j x_{i+1, h} \geq  \sum\limits_{h=1}^j x_{ih} \, \mbox{for all } i \in [p-1], \, j \in [q-1].$
\end{enumerate}
To prove that the inequalities (d1), (d2), (d3a), and (d3b) constitute the minimal $H$-representation of $\UDC_{p,q}$, we first demonstrate that if $[c_{ij}]\in\R^{(p+1)\times(q+1)}$ satisfies the boundary condition (c1) and all of (d1), (d2), (d3a), and (d3b), then $[c_{ij}]\in\UDC_{p,q}$.  
This is proven in Lemma~\ref{Lem:1} in the Appendix.  
Then we show that for each inequality in the list (d1), (d2), (d3a), and (d3b) there exists a point $[c_{ij}]\in\R^{(p+1)\times(q+1)}$ failing to satisfy this inequality that satisfies all the other inequalities.
We do this by proving the analogous fact for the subpolytope $\TUDC_{p,q}$ of $\B_{p,q}$.  
Since the details of this argument are technical, the complete proof is given in the Appendix.  

In the following theorem and remark we show that every point in $\UDC_{p,q}$ can be realized as a restriction of some ultramodular bivariate copula on $[0,1]^2$ and that any restriction of an ultramodular discrete copula is in fact a point in $\UDC_{p,q}$.
In particular, any point in $\UDC_{p,q}$ can be extended to an ultramodular copula on $[0,1]^2$ known as the \emph{checkerboard extension copula} \citep{nelsen_06}. The proof of Theorem~\ref{Th:ext} is given in the Appendix.
%---THEOREM: Checkerboard Extension Theorem----
\begin{theorem}
\label{Th:ext}
Given $p,q\in \Z_{>0}$, the checkerboard extension copula of any $[c_{ij}] \in \UDC_{p,q}$ is an ultramodular copula on the unit square.
\end{theorem}
%---REMARK:  Restriction of an Ultramodular Copula is always a point in UDC_{p,q}----
\begin{remark}
\label{Rem:UC}
The restriction $C$ of any ultramodular copula $\tilde{C}$ on a non-square uniform grid $I_p \times I_q$ of the unit square belongs to $\UDC_{p,q}$.
Indeed, let us consider a copula $\tilde{C}$ that is ultramodular. 
Then the restriction $C$ of $\tilde{C}$ to the interval $I_p \times I_q$ is a discrete copula \citep{kolesarova_2006,nelsen_06}.
Therefore, $C$ belongs to $\DC_{p,q}$ and satisfies (d1), (d2), and (d3).
Since $\tilde{C}$ is ultramodular, all of its horizontal and vertical sections are univariate continuous convex functions that fulfill the \emph{Jensen inequality}; i.e., for every $u_1,u_2 \in [0,1]$, and $a \in [0,1]$, 
$
\tilde{C}\big(\tfrac{u_1}{2} + \tfrac{u_2}{2}, a \big) \leq  \tfrac{1}{2} \tilde{C}(u_1,a) + \tfrac{1}{2} \tilde{C}(u_2,a) .
$
Inequalities (d3b) can be derived by fixing $a= \tfrac{j}{q}$, while $u_1 = \tfrac{i}{p}$, $u_2 = \tfrac{i+2}{p}$ for $j\in [q-1]$ and $i \in \langle p-2\rangle$. 
In an analogous manner, one can obtain conditions (d3a).
Hence, $C \in \UDC_{p,q}$. \qed
\end{remark}

Theorem~\ref{Th:ext} and Remark~\ref{Rem:UC} also provide a statistical interpretation for the polytope $\TUDC_{p,q}$.
In particular, they identify a correspondence between each point in $\TUDC_{p,q}$, normalized with a multiplicative factor $\frac{1}{pq}$, and the probability mass of an ultramodular bivariate copula on $[0,1]^2$, which can be constructed via checkerboard extension techniques. 
This is interesting from a statistical perspective as the checkerboard extension copula plays a crucial role  in the entropy-copula approaches presented in \cite{piantadosi_2007,piantadosi_2012}, and it is at the base of the empirical multilinear copula process recently introduced by \citet{genest_2014,genest_2017}.

%---SECTION: The Polytopes of Discrete Quasi-copulas and Convex Discrete Quasi-copulas----
\section{Polytopes of (Convex) Discrete Quasi-copulas}
\label{sec: the polytopes of discrete quasi-copulas and convex discrete quasi-copulas}
In this section, we identify the minimal $H$-representations for the polytope of discrete quasi-copulas $\DQ_{p,q}$ and its subpolytope of convex discrete quasi-copulas $\CDQ_{p,q}$.  
Recall from Proposition~\ref{prop: unimodular equivalence} that $\DQ_{p,q}$ is unimodularly equivalent to a dilation of the generalized alternating sign matrix polytope $\ASM_{p,q}$, which was originally studied in \cite[Chapter 5]{Knight_2009}.  
However, while the minimal $H$-representation for the case $p=q$ (i.e., for the polytope $\ASM_p$) was identified in \cite[Theorem 3.3]{striker_2009}, it was unknown for $p\neq q$.  
In this section, we identify the minimal $H$-representation for $\ASM_{p,q}$ (and hence also for $\DQ_{p,q}$) and also for the polytope $\CDQ_{p,q}$.  

It is shown in \cite[Theorem 3.3]{striker_2009} that for $p\geq 3$ the polytope $\ASM_p$ has $4[(p-2)^2+1]$ facets given by the inequalities
\begin{enumerate}
	\item[(1)] $x_{11} \geq 0, \quad x_{1p} \geq 0, \quad x_{p1} \geq 0, \quad \mbox{and }\; x_{pp} \geq 0;$
	\vspace{3pt}
	\item [(2)] $\sum_{k = 1}^{i-1} x_{kj} \geq 0,  \mbox{ and } \sum_{k = i+1}^px_{kj}\geq 0 \mbox{ for } i,j \in\{2,\ldots,p-1\};$
		\vspace{3pt}
	\item[(3)] $\sum_{h = 1}^{j-1} x_{ih} \geq 0, \mbox{ and } \sum_{h = j+1}^px_{ih}\geq 0 \mbox{ for }  i,j \in\{2,\ldots,p-1\}.$
\end{enumerate}
Suppose now that ${3 \leq} p<q$ and that $q = kp+r$ for $0\leq r < p$.Our second main theorem of the paper generalizes Theorem 3.3 of \cite{striker_2009}.
%---THEOREM: Facets of the Generalized Alternating Sign Matrix Polytope----
\begin{theorem}
\label{thm: facets of the generalized alternating sign matrix polytope}
Suppose $3 \leq p<q$ with $q = kp+r$ for $0\leq r < p$.  
The minimal $H$-representation of the generalized alternating sign matrix polytope $\ASM_{p,q}$ consists of the 
$
2((p-1)(q-2)+2)+2(p-2)(q-k-1)
$
inequalities
\begin{enumerate}
	\item[(a1)] $x_{11} \geq 0, \quad x_{1q} \geq 0, \quad x_{p1} \geq 0, \quad \mbox{and }\; x_{pq} \geq 0;$
	\vspace{3pt}
	\item [(a2)] $\sum_{\ell = 1}^{i-1} x_{\ell j} \geq 0,\;\; \sum_{\ell = i+1}^px_{\ell j}\geq 0 \; \mbox{ for } i\in[p-1], \, j \in\{2,\ldots,q-1\};$
	\vspace{3pt}
	\item[(a3)] $\sum_{h = 1}^{j-1} x_{ih} \geq 0,\;\;  \sum_{h = j+1}^px_{ih}\geq 0\; \mbox{ for }  i\in\{2,\ldots,p-1\}, \, j\in[q-k-1].$
\end{enumerate}
\end{theorem}
To proof is given in the Appendix and is analogous to the approach taken for proving Theorem~\ref{Th:h-repUC}.   
The natural functional generalization of ultramodular discrete copulas to the setting of quasi-copulas are \emph{convex} discrete quasi-copulas; i.e., discrete quasi-copulas admitting convex (coordinatewise) sections.  
These functions are parametrized by the points $[c_{ij}]$ within the polytope $\CDQ_{p,q}$, which has the following H-representation:  
%
%---THEOREM: Minimal H-representation of the Polytope of Convex Discrete Quasi-copulas---
\begin{theorem}
\label{Th:h-repCDQC}
The minimal $H$-representation of the polytope of convex discrete quasi-copulas $\CDQ_{p,q}$ consists of the 
$
2[(p-1)(q-1) + 1]
$
inequalities
\begin{enumerate}
	\item[(v1)] $x_{11}\geq 0, \quad x_{p-1,q-1} \geq \tfrac{(p-1)(q-1)-1}{pq};$
	\vspace{3pt}
	\item [(v3a)] $x_{ij} + x_{i,j+2} - 2\, x_{i,j+1} \geq 0  \quad \mbox{ for all } i \in [p-1], \; j \in \langle q-2\rangle;$
	\vspace{3pt}
	\item[(v3b)] $x_{ij} + x_{i+2,j} - 2\, x_{i+1, j} \geq 0  \quad \mbox{ for all } j \in [q-1], \; i \in \langle p-2\rangle.$
\end{enumerate}
\end{theorem}
The proof is again analogous to the proof of Theorem~\ref{Th:h-repUC}, and is given in the Appendix.  
In particular, in the proof we show that the unimodularly equivalent subpolytope $\TCDQ_{p,q}$ of $\ASM_{p,q}$ has minimal $H$-representation
\begin{enumerate}
	\item[(a1)] $x_{11}\geq 0, \quad \, x_{pq} \geq 0;$
	\vspace{3pt}
	\item [(a3a)] $\sum\limits_{\ell=1}^i x_{\ell,j+1} \geq  \sum\limits_{\ell=1}^i x_{\ell j} \, \mbox{ for all } i \in [p-1], \, j \in [q-1];$
	\vspace{3pt}
	\item[(a3b)] $\sum\limits_{h=1}^j x_{i+1, h} \geq  \sum\limits_{h=1}^j x_{ih} \, \mbox{ for all } i \in [p-1], \, j \in [q-1].$
\end{enumerate}
Since convex discrete quasi-copulas are the natural generalization of ultramodular discrete copulas to the quasi-copula setting, we would hope that the points $[c_{ij}]\in\CDQ_{p,q}$ are, analogously, the family of points that can be extended to convex quasi-copulas on $[0,1]^2$.  
Indeed, this is the case: 
%---THEOREM: Extending convex discrete quasi-copulas----
\begin{theorem}
\label{Th:extQC}
Given $p,q\in\Z_{>0}$, the checkerboard extension of any $[c_{ij}]\in\CDQ_{p,q}$ is a quasi-copula on $[0,1]^2$ with convex (coordinatewise) sections.
\end{theorem}
%---REMARK: Restricting Convex Discrete Quasi-copulas----
\begin{remark}
\label{Rem:CDQ}
Following the same considerations as in Remark~\ref{Rem:UC}, one can notice that the restriction $C$ of any quasi-copula $\tilde{C}$ on a non-square uniform grid $I_p \times I_q$ of the unit square belongs to $\CDQ_{p,q}$. \qed
\end{remark}

Analogous to the case of ultramodular copulas, it is useful to notice that Theorem~\ref{Th:extQC} and Remark~\ref{Rem:CDQ} identify a correspondence between each point in $\TCDQ_{p,q}$, normalized with a multiplicative factor $\frac{1}{pq}$, and the signed measure of a bivariate quasi-copula with convex sections.
Interestingly, the family of quasi-copulas with convex horizontal and vertical sections has not been studied before. 
Our findings suggest that further research efforts should be made in understanding the properties of this class of quasi-copulas and its relation to ultramodular copulas.

%---SECTION: On Vertex Representations-----
\section{On Vertex Representations}
\label{sec: on vertex representations}
In the previous sections we showed that two special families of discrete copulas and discrete quasi-copulas admit representations as convex polytopes using collections of inequalities.  A powerful feature of working with convex polytopes is that they admit an alternative representation as the convex hull of their vertices (i.e., \emph{extreme points}).  
If $S\subset\R^p$ then the \emph{convex hull} of $S$, denoted $\conv(S)$, is the collection of all convex combinations of points in $S$.  
A point $x\in S$ is called an \emph{extreme point} of $S$ provided that for any two points $a,b\in S$ for which $(a+b)/2 = x$, we have that $a = b = x$.  
If $P\subset\R^p$ is a convex polytope, an extreme point of $P$ is called a \emph{vertex} and the collection of all vertices of $P$ is denoted $\V(P)$.  
The \emph{Krein-Milman Theorem} in convex geometry \citep[Theorem 3.3]{Barvinok_02} states that $P$ can be represented by its collection of vertices, namely $P = \conv(\V(P))$.
The collection of vertices of a convex polytope is known as its \emph{$V$-representation}.

For example, the vertices of the Birkhoff polytope $\B_p$ are precisely the $p\times p$ permutation matrices (see for instance \cite[Theorem 5.2]{Barvinok_02}). 
In \cite{piantadosi_2007}, the $V$-representation of the Birkhoff polytope is used to efficiently find a bistochastic matrix $B$ representing a joint density that matches a prescribed grade correlation coefficient and maximizes the entropy. 
The full-domain checkerboard extension copula of $B$ is the one with maximum entropy, and can be used to conduct further statistical analysis avoiding additional model assumptions.
In the setting of discrete copulas, the vertices of $\DC_p$ correspond to the \emph{empirical copulas} \citep{kolesarova_2006,Mesiar_2005}, and thus all bivariate discrete copulas can be constructed by way of convex combinations of empirical copulas. 
This $V$-representation of $\DC_p$ is known in the statistical literature: the empirical copulas are precisely the copulas constructible from observed data \citep{Mesiar_2005}, which has made them fundamental in the development of \emph{rank-based copula methods} \citep{joe_2014,Scaillet2007}.

Thus, if a family of discrete copulas or quasi-copulas admits a representation as a polytope, it may be beneficial to identify its $V$-representation. 
At the same time, polytopes can often have a super-exponential number of vertices, meaning that it may be difficult to learn its $V$-representation in its entirety.
Indeed, this appears to be the case for the polytopes $\UDC_{p,q}$ and $\CDQ_{p,q}$, as suggested by the data in Table~\ref{tab: vertices}. 
Although complete $V$-representations of $\UDC_{p,q}$ and $\CDQ_{p,q}$ seem out of reach, we can still benefit from knowing the vertices of $\UDC_{p,q}$ for $(p,q)=\{(3,3),(4,4)\}$ to possibly select ultramodular copulas with maximum entropy in a similar fashion as in \cite{piantadosi_2007}. 
Therefore, in the following subsection we provide two constructions to obtain families of vertices for each of these polytopes.

%---TABLE: VERTICES----
\begin{table}
\centering
\label{tab: vertices}
\begin{tabular}{|c||cccc|}
 \hline
 $\bf{(p,q)}$& \textbf{UDC} & \textbf{CDQ} & \textbf{DQ} & \textbf{DC}  \\
 \hline
 $(3,3)$& \textcolor{black}{7} & \textcolor{black}{7} & 7 & 6 \\ 
 \hline
 $(3,4)$& \textcolor{black}{52}&   \textcolor{black}{52} & 118 & 96  \\
  \hline
 $(3,5)$& \textcolor{black}{166}&   \textcolor{black}{138}  & 416 & 360  \\
 \hline
 $(4,4)$& \textcolor{black}{115}&  \textcolor{black}{69} & 42 & 24 \\
  \hline
  $(4,5)$& \textcolor{black}{3321}&  \textcolor{black}{2163} & 7636 & 3000 \\
  \hline
 $(5,5)$& \textcolor{black}{22890}&   \textcolor{black}{5447}  & 429 & 120  \\
\hline
\end{tabular}
\label{tab:3}
\caption{
The number of vertices of $\UDC_{p,q}, \CDQ_{p,q}, \DQ_{p,q}$, and $\DC_{p,q}$ as computed using {\tt polymake} \citep{polymake:2000}.}
\end{table}

%---SUBSECTION: New Vertices by Way of Symmetry----
\subsection{New Vertices by Way of Symmetry.}
\label{subsec: new vertices by way of symmetry}
In the coming subsections, we will use the following fundamental theorem from convex geometry:  
% 
%---THEOREM: Fundamental Theorem in Convex Geometry----
\begin{theorem}
\cite[Theorem 4.2]{Barvinok_02}
\label{thm: vertices}
Let $P := \{
x\in\R^p :
\langle a_i,x\rangle \leq\beta_i 
\mbox{ for $i\in[m]$}
\}$ be a polyhedron, 
where $a_i\in\R^p$ and $\beta_i\in\R$ for $i\in[m]$.  
For $u\in P$ let $I(u) :=
\{ 
i\in[m] :
\langle a_i, u\rangle = \beta_i
\}$
be the collection of inequalities that are \emph{active} on $u$.  
Then $u$ is a vertex of $P$ if and only if the set of vectors 
$
\{
a_i :
i\in I(u)
\}
$
linearly spans the vector space $\R^p$.  
In particular,
if $u$ is a vertex of $P$, then the set $I(u)$ contains at least $p$ indices, i.e.,  $|I(u)|\geq p$.  
\end{theorem}

The vectors $\beta_i$ in the above theorem are called the \emph{facet-normals} or \emph{facet-normal vectors} of the polyhedron $P$.  
We now apply Theorem~\ref{thm: vertices} to prove a basic symmetry statement about the vertices of $\UDC_{p,q}$ and $\CDQ_{p,q}$.   
Recall that we think of a bivariate discrete (quasi)-copula 
$
C:
I_p\times I_q \longrightarrow [0,1]
$
as a $(p+1) \times (q+1)$ matrix $C = [c_{ij}]_{i,j=0}^{p,q}$ whose entries are the values of $C$.  
Given this representation for $C$, we can then consider its transpose $C^T$.  
%---PROPOSITION: Reflections of Vertices----
\begin{proposition}
\label{prop: reflections of vertices}
Suppose that $C\in\UDC_{p,q}$ $(C\in\CDQ_{p,q})$, then $C^T\in\UDC_{q,p}$ $(C^T\in\CDQ_{q,p})$. 
Moreover, if $C$ is a vertex of $\UDC_{p,q}$ $(\CDQ_{p,q})$, then $C^T$ is a vertex of $\UDC_{q,p}$ $(CDQ_{q,p})$.
\end{proposition}

\begin{proof}
We here prove the statement for the polytope $\UDC_{p,q}$.  
The proof for $\CDQ_{p,q}$ works analogously.  
Recall that the facet-defining inequalities for $\UDC_{p,q}$ are (d1), (d2), (d3a), and (d3b), which can be reorganized as:
\begin{enumerate}
	\item[(1)] $x_{11}\geq 0$,
	\item[(2)] $x_{p-1,q-1}\geq\frac{(p-1)(q-1)-1}{pq}$,
	\item[(3)] 
	$
	x_{ij}+x_{i+1,j+1} \geq x_{i,j+1}+x_{i+1,j}
	$
	for $i\in[p-2]$ and $j\in[q-2]$,
	\item[(4)] 
	$
	2x_{ij}\leq x_{i-1,j}+x_{i+1,j}
	$
	and
	$
	2x_{ij}\leq x_{i,j-1}+x_{i,j+1}
	$
	for $i\in[p-1]$ and $j\in[q-1]$.
\end{enumerate}
By Theorem~\ref{Th:h-repUC}, the minimal H-representation of $\UDC_{p,q}$ is given by the inequalities (1), (2), (3), and (4) with the exception of the two inequalities
$
x_{11}+x_{22}\geq x_{12}+x_{21}
$
and
$
x_{p-2,q-2}+x_{p-1,q-1}\geq x_{p-2,q-1}+x_{p-1,q-2}.
$
From this presentation of the minimal H-representation of $\UDC_{p,q}$, we can see that if $C\in\UDC_{p,q}$ then $C^T\in\UDC_{q,p}$.  
Moreover, by Theorem~\ref{thm: vertices}, it follows that if $C$ is a vertex of $\UDC_{p,q}$ then $C^T$ is a vertex of $\UDC_{q,p}$.  
\end{proof}
The vertex construction technique of Proposition~\ref{prop: reflections of vertices} suggests that the most informative extremal discrete copulas of $\UDC_{p}$ are those $\tilde{C}=[c_{ij}]$ such that $c_{ij} \neq c_{ji}$, for some $i, j \in \langle p \rangle$. Indeed, the transpose of any such $\tilde{C}$ is a new distinct vertex of $\UDC_{p}$. Thus, the checkerboard extension copulas constructed from any such vertex $\tilde{C}$ are asymmetric copulas, i.e. those that describe the stochastic dependence of non-exchangeable random variables.
Furthermore, Proposition~\ref{prop: reflections of vertices} is an intuitive result whose proof provides a nice example of how Theorem~\ref{thm: vertices}  can be used to study extremal discrete copulas.  Another example is provided in the following subsection.

%---SUBSECTION: New Vertices by Way of Direct Products----
\subsection{New vertices by way of direct products.}
\label{subsec: new vertices by way of direct products}
Our second family of vertices arises by taking direct sums of lower-dimensional vertices.  
Recall from Proposition~\ref{prop: unimodular equivalence} that there is a linear map $T: \R^{(p+1) \times (q+1)}\longrightarrow \R^{p\times q}$ sending a discrete (quasi)-copula to a matrix in $\frac{1}{pq}\B_{p,q}$ (a matrix in $\frac{1}{pq}\ASM_{p,q}$).  
Further recall that $\TUDC_{p,q} = pqT(\UDC_{p,q})$ and $\TCDQ_{p,q} = pqT(\CDQ_{p,q})$.  
Define the direct sum of $B\in\TUDC_{p,q}$ $(\TCDQ_{p,q})$ and $D\in\TUDC_{s,t}$ $(\TCDQ_{s,t})$ to be the block matrix
\[
B\oplus D := 
\begin{pmatrix}
{\bf 0}_{p,t}	&	B	\\
D			&	{\bf 0}_{s,q}	\\
\end{pmatrix} \in \R^{(p+s) \times (q+t)}.
\]
Indeed, if we applied the transformation $R :\R^{(p+s) \times (q+t)} \longrightarrow \R^{(p+s) \times (q+t)}$ with $e_{ij} \longmapsto e_{i(q+t-j+1)}$, then $R(B\oplus D)$ is the direct sum of $R(B)$ and $R(D)$.  
In the following, we show how to use this operation to identify vertices of $\TUDC_{p,q}$ and $\TCDQ_{p,q}$ (and equivalently $\UDC_{p,q}$ and $\CDQ_{p,q}$).  

Recall from Section~\ref{sec: copulas and quasi-copulas in discrete geometry} that $\mathcal{T}(u,v)$ denotes the transportation polytope with marginals $u\in\R^p$ and $v\in\R^q$, and $\mathcal{A}(u,v)$ denotes the alternating transportation polytope with the same marginals.  
The subpolytopes $\TUDC_{p,q}\subset\B_{p,q}$ and $\TCDQ_{p,q}\subset\ASM_{p,q}$ admit a natural geometric generalization to subpolytopes $\TUDC(u,v)\subset \mathcal{T}(u,v)$ and $\TCDQ(u,v)\subset \mathcal{A}(u,v)$.  
Namely, we let $\TUDC(u,v)$ denote the subpolytope of $\mathcal{T}(u,v)$ satisfying the additional inequalities (b3a) and (b3b), and we let $\TCDQ(u,v)$ denote the subpolytope of $\mathcal{A}(u,v)$ satisfying the additional inequalities (a3a) and (a3b).  
In the following, for $m,k\in\Z$, let ${\bf m}_p := (m,m,\ldots,m)\in\R^p$, and let $({\bf m}_p,{\bf k}_q)\in\R^{p+q}$ denote the concatenation of the vectors ${\bf m}_p$ and ${\bf k}_q$.  
We can then make the following geometric observation.
%---THEOREM: Recursively Constructed Vertices-----
\begin{theorem}
\label{thm: recursively constructed vertices}
If $B$ is a vertex of $\TUDC_{p,q}$ $(\TCDQ_{p,q})$ and $D$ is a vertex of $\TUDC_{s,t}$ $(\TCDQ_{s,t})$, then $B\oplus D$ is a vertex of $\TUDC(({\bf q}_p,{\bf t}_s),({\bf s}_t,{\bf p}_q))$
(and analogously, $\TCDQ(({\bf q}_p,{\bf t}_s),({\bf s}_t,{\bf p}_q))$).  
\end{theorem}

The proof of this result is another application of Theorem~\ref{thm: vertices} which is given in the Appendix. In the special case  where $p = q$ and $s = t$, then $\TUDC_{p,q}$ and $\TUDC_{s,t}$ are dilations of subpolytopes of $\B_p$ and $\B_s$, respectively. 
Thus, we can assume that the marginals of $\mathcal{T}(u,v)$ are $u = v = {\bf 1}_{p+s}\in\R^{p+s}$.  
Therefore, Theorem~\ref{thm: recursively constructed vertices} produces vertices of $\UDC_{p+s}$. The proof of the following corollary can be found in the Appendix. 
%---COROLLARY: Vertices when p=q----
\begin{corollary}
\label{cor: vertices when p=q}
If $B$ is a vertex of $\TUDC_{p}$ $(\TCDQ_{p})$ and $D$ is a vertex of $\TUDC_{s}$ $(\TCDQ_{s})$, then $B\oplus D$ is a vertex of $\TUDC_{p+s}$ ($\TCDQ_{p+s}$).  
\end{corollary}

%---REMARK: Statistical Interpretation of Vertices----
\begin{remark}[Statistical Interpretation of Vertices]
\label{rmk: statistical interpretation of vertices}
Given a copula $C$, a \emph{patchwork copula} derived from $C$ is any copula whose probability distribution coincides with the one of $C$ up to a finite number of rectangles $R_i$ in $[0,1]^2$ in which the probability mass is distributed differently \citep{Durante_13}.
The vertices obtained via Corollary~\ref{cor: vertices when p=q} correspond to a special class of patchwork (quasi)-copulas named \emph{W-ordinal sums}, which are patchworks derived from the Fr\'echet lower bound of copulas $W(u,v)=\max \{0,u+v-1\}$ \citep{Mesiar_2004}. 
The (normalized) direct sum of two vertices $B\in \TUDC_{p}$ $(\TCDQ_{p})$ and $ D\in \TUDC_{s}$ $(\TCDQ_{s})$ is the  block matrix $\frac{1}{p+q}B\oplus D$
%\[
%B\oplus D := \tfrac{1}{p+s}
%\begin{pmatrix}
%{\bf 0}_{p,s}	&	 B	\\
% D			&	{\bf 0}_{s,p}	\\
%\end{pmatrix} \in \R^{(p+s) \times (p+s)}.
%\]
Any extension (quasi)-copula $\tilde{C}$ on $[0,1]^2$, whose associated mass is given by $\frac{1}{p+q}B\oplus D$, satisfies $\tilde{C}\left(\tfrac{p}{p+s},\tfrac{s}{p+s}\right)=0$. 
Furthermore, any (quasi)-copula $C$ with $C(u_0, 1-u_0)=0$ for $ 0 < u_0 < 1$ can be written as a W-ordinal sum \citep{De_Baets_2007}. 
Thus, any such $\tilde{C}$ associated to $\frac{1}{p+q}B\oplus D$ is a W-ordinal sum. \qed
\end{remark}

While Corollary~\ref{cor: vertices when p=q} is a useful method for constructing vertices of $\UDC_p$ and $\CDQ_p$ from known, lower-dimensional vertices, it is important to notice that not all vertices of $\UDC_p$ and $\CDQ_p$ can be captured in this fashion.  
For example, as we can see in Figure~\ref{fig: vertices-example}, $\UDC_3$ has seven vertices, of which only three arise from this direct sum construction.  
However, as we show in the following subsection, Corollary~\ref{cor: vertices when p=q} can be used to provide lower bounds on the number of vertices of these polytopes.

%---SUBSECTION: On Generating Functions for the Number of Vertices----
\subsection{Generating Functions for the Number of Vertices.}
\label{subsec: on generating functions for the number of vertices}
In this subsection we consider the special case of the polytopes $\UDC_{p,q}$ and $\CDQ_{p,q}$ for which $p = q$.  
For convenience, we only discuss the polytope $\UDC_p$.  
However, the results all hold analogously for $\CDQ_p$.  
Corollary~\ref{cor: vertices when p=q} gives a convenient way by which to partition the collection of vertices $\V(\UDC_{p})$ into two disjoint collections: we call a vertex of $\UDC_p$ \emph{decomposable} if the corresponding vertex in $\TUDC_p$ admits a decomposition as a direct sum of two lower dimensional vertices as in Corollary~\ref{cor: vertices when p=q}.  
All other vertices of $\UDC_p$ are called \emph{indecomposable}.   
Let $\D_p$ and $\ID_p$ denote the decomposable and indecomposable vertices of $\UDC_p$, respectively, and let
\[
V(x) :=\sum_{p\geq0}|\V(UDC_p)|x^p, \, \, \ID(x) := \sum_{p\geq 0}|\ID_p|x^p, \, \,  \mbox{and } \D(x) := \sum_{p\geq 0}|\D_p|x^p,
\]
denote the \emph{generating functions} for the values $|\V(\UDC_p)|$, $|\ID_p|$, and $|\D_p|$, respectively.  As suggested by the data in Table~\ref{tab: vertices}, the size of the set $\V(\UDC_p)$ appears to grow super-exponentially in $p$. The following observation, whose proof is given in the Appendix, may be used to provide lower bounds supporting this observed growth-rate.  

%---FIGURE: UDC_3 Edge Graph-----------
\begin{figure}[t!]
	\centering
	$\begin{array}{c c c}
	\includegraphics[width=0.4\textwidth]{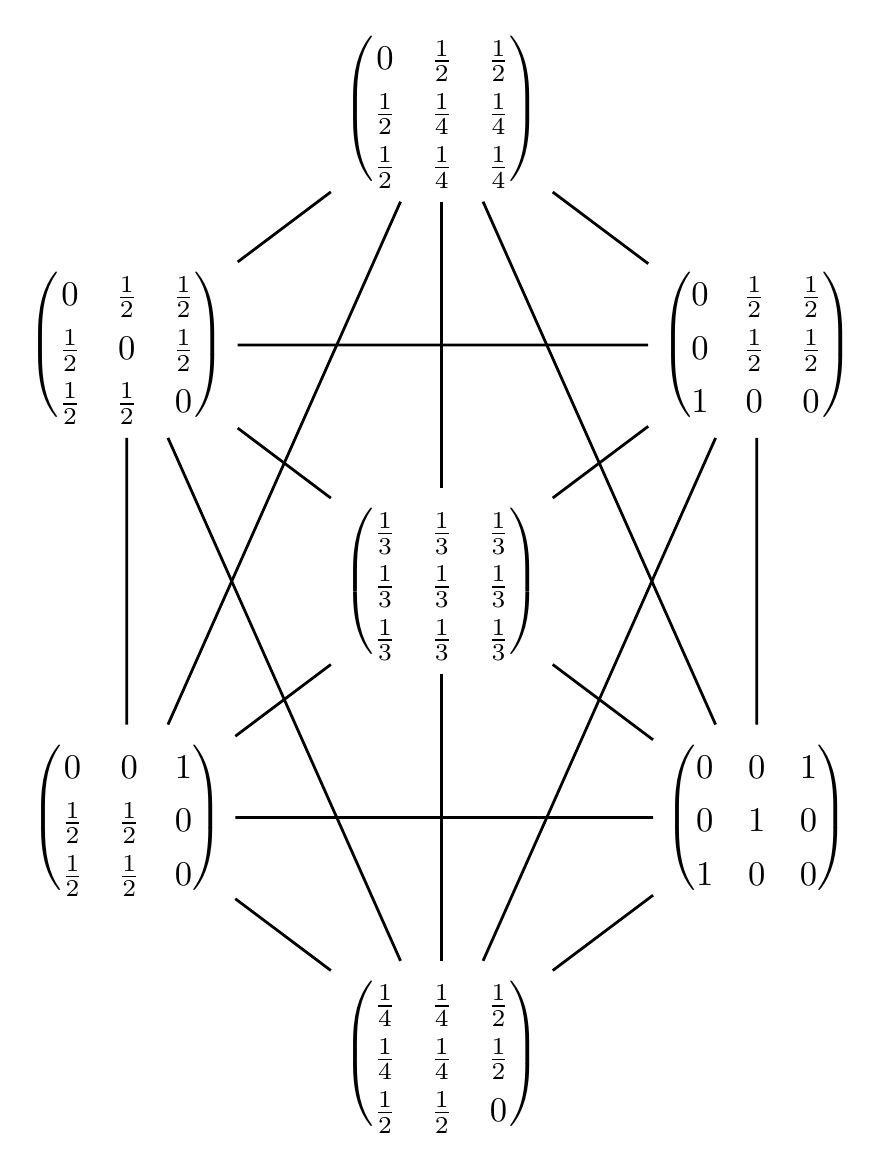}	& \hspace{15pt} &	\includegraphics[width=0.45\textwidth]{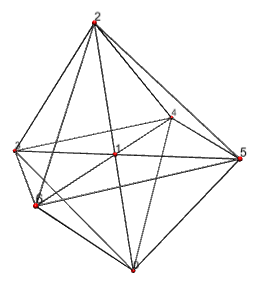}	
	\end{array}$
	\vspace{-0.4cm}
	\caption{The edge-graph of the polytope $\TUDC_3$, with its seven vertices, is the edge graph of a triangulated octahedron.  Indeed, $\TUDC_{3}$ is a four-dimensional polytope, with eight simplicial facets and one octahedral facet.  On the right is a \emph{Schlegel diagram} \citep{Ziegler_95} of $\TUDC_{3}$ as it appears when projected onto its three-dimensional, octahedral facet.  }
	\label{fig: vertices-example}
\end{figure}

%---PROPOSITION: Generating Functions---
\begin{proposition}
\label{prop: generating functions}
The number of vertices of $\UDC_p$ is computable in terms of its number of decomposable vertices by the relationship
\[
V(x) = \frac{D(x)^2 + D(x) -1}{D(x)}.
\]
Moreover, if $M(x)\leq D(x)$, is a lower-bound on the number of decomposable vertices of $\UDC_p$ then 
$
V(x) \geq (M(x)^2 + M(x) -1)/M(x).
$
\end{proposition}
Since a lower bound on the number of decomposable vertices can be achieved by identifying a lower bound on the number of indecomposable vertices in lower dimensions, it is worthwhile to investigate large families of indecomposable \emph{extremal} ultramodular discrete copulas.  
The identification of sufficiently large families of such copulas could then be used to prove that the size of the vertex representation of $\UDC_p$ grows super-exponentially, as well as serve to generate larger families of vertices of these polytopes for statistical use by the construction given in Corollary~\ref{cor: vertices when p=q}.

%---SECTION: Aggregation Functions and Alternating Transportation Polytopes----
\section{Aggregation Functions \& Alternating Transportation Polytopes}
\label{sec: aggregation functions and alternating transportation polytopes}
We end this paper with a discussion aimed at completing the evolving parallel story between discrete bivariate copulas, the Birkhoff polytopes and each of their generalizations.  
In Section~\ref{sec: copulas and quasi-copulas in discrete geometry}, we highlighted the following hierarchy of generalizations of Birkhoff polytopes:
\[
\begin{array}{ccccc}
\shortstack{\mbox{Birkhoff}\\\mbox{Polytopes}}	&	\subset  	&	\shortstack{\mbox{Generalized} \\\mbox{Birkhoff Polytopes} }	&	\subset  	&	\shortstack{\mbox{Transportation}\\
\mbox{Polytopes}} \\
\cap	&	&	\cap	&	&	\cap	\\
\shortstack{\mbox{Alternating Sign}\\\mbox{Matrix Polytopes}}	&	\subset  	&	\shortstack{\mbox{Generalized Alternating} \\\mbox{Sign Matrix Polytopes} }	&	\subset  	&	\shortstack{\mbox{Alternating}\\
\mbox{Transportation Polytopes}} \\
\end{array}
\]
Analogously, we have the hierarchy of generalizations of discrete copulas:
\[
\begin{array}{ccccc}
\shortstack{\mbox{$p\times p$ Discrete}\\\mbox{Copulas}}	&	\subset  	&	\shortstack{\mbox{$p\times q$ Discrete} \\\mbox{Copulas} }	&	\subset  	&	?	 \\
\cap	&	&	\cap	&	&	\cap	\\
\shortstack{\mbox{$p\times p$ Discrete }\\\mbox{Quasi-copulas}}	&	\subset  	&	\shortstack{\mbox{$p\times q$ Discrete } \\\mbox{Quasi-copulas} }	&	\subset  	&	?	\\
\end{array}
\]
The main efforts of this paper were aimed at identifying polyhedral representations of subfamilies of each of these collections of functions (Sections~\ref{sec: the polytope of ultramodular discrete copulas} and~\ref{sec: the polytopes of discrete quasi-copulas and convex discrete quasi-copulas}) as well as a polyhedral representation of the family of $p\times q$ discrete quasi-copulas in its entirety (Theorem~\ref{thm: facets of the generalized alternating sign matrix polytope}).  
However, we can also extend the correspondence between these hierarchies of generalizations in terms of a functional generalization of copulas:
%---DEFINITION: Binary Aggregation Functions-----
\begin{definition}\citep{Grabisch_2009}
\label{Def:AggrFun}
A \emph{(binary) aggregation function} is a function $C:[0,1]^2 \rightarrow [0,1]$ that satisfies the following
\begin{enumerate}
\item[(A1)] $C(0,0)=0$ \; and\; $C(1,1)=1$;
\vspace{3pt}
\item[(A2)] $C(u_1,v_1) \leq C(u_2,v_2)$,\; for every \;$ u_1\leq u_2, v_1\leq v_2 \in [0,1]$.
\end{enumerate}
\end{definition}

Aggregation functions naturally include copulas and quasi-copulas. 
In particular, copulas are the supermodular aggregation functions with annihilator 0 and neutral element 1, and quasi-copulas are 1-Lipschitz aggregation functions with annihilator 0 and neutral element 1.
Analogous to the case of discrete (quasi)-copulas, we can consider discrete aggregation functions $C_{pq}$ with domain $I_p \times I_q$, which are representable by a matrix $[c_{ij}]\in\R^{(p+1)\times(q+1)}$ where $c_{ij}:=C_{p,q}(i/p,j/q)$.
By way of the same linear transformation used in Proposition~\ref{prop: unimodular equivalence}, we now observe a correspondence between discrete aggregation functions and the matrices within (alternating) transportation polytopes $\mathcal{A}(u,v)$ with \emph{homogeneous marginals}; i.e., $\sum_i u_i = \sum_j v_j = pq$.  

Given two vectors $\tilde{u}:=(\tilde{u}_1, \ldots, \tilde{u}_p) \in \mathbb{R}^p_{>0}$ and $\tilde{v}:=(\tilde{v}_1, \ldots, \tilde{v}_q) \in \mathbb{R}^q_{>0}$ with $\tilde{u}_{p} = \tilde{v}_{q} = pq$, we define the set $\text{SAF}(\tilde{u},\tilde{v})$ to be the matrices $[c_{ij}]\in\R^{(p+1)\times(q+1)}$ satisfying
\begin{enumerate}
\item[(AF1a)] $c_{0j}=0$,\; $c_{i0}=0$\; with \;$i \in \langle p \rangle$, $j \in \langle q \rangle$;
\vspace{3pt}
\item[(AF1b)] $c_{p,j-1}=\tfrac{\tilde{v}_{j-1}}{pq} < \tfrac{\tilde{v}_{j}}{pq}=c_{pj}$,\; $c_{i-1,q}=\tfrac{\tilde{u}_{i-1}}{pq} < \tfrac{\tilde{u}_{i}}{pq}=c_{iq}$,
\vspace{3pt}
\item[(AF2a)] $c_{ij} + c_{i-1,j-1} - c_{i-1,j} - c_{i,j-1} \geq 0$\; for every \;$i \in [p]$, $j \in [q]$.
\end{enumerate}

Note that the elements of $\text{SAF}(\tilde{u},\tilde{v})$ are discrete aggregation functions. 
Indeed, given any $i_1 \leq i_2 \in \langle p \rangle $, and $j_1 \leq j_2 \in \langle p \rangle $, property (AF2a) implies that $c_{i_2,j_2} \geq c_{i_1,j_2} \geq c_{i_1,j_1}$.
The following proposition links the set $\text{SAF}(\tilde{u},\tilde{v})$ to a transportation polytope $\mathcal{T}(u,v)$ with homogeneous marginals.
The proof of the following two propositions can be found in the Appendix.
%-PROPOSITION: Transportation Matrices and Aggregation Functions----
\begin{proposition}
\label{Prop:AggrFuncTranPol}
For a function $C_{pq}:I_p \times I_q \rightarrow [0,1]$, the following statements are equivalent:
\begin{enumerate}
\item[(i)] $C_{pq} \in \text{SAF}(\tilde{u},\tilde{v})$.
\vspace{3pt}
\item[(ii)] There is a $(p \times q)$ transportation matrix $[x_{ij}]$ in $\mathcal{T}(u,v)$, with $\sum_{h=1}^{q} v_h = \sum_{\ell=1}^{p} u_{\ell} = pq$,  such that for every $i \in \langle p \rangle$, $j \in \langle q \rangle$ 
\begin{equation}\label{Eq:SAF_TP}
c_{ij}:=C_{p,q}\left(\frac{i}{p},\frac{j}{q}\right)= \frac{1}{pq} \sum_{\ell =1}^i \sum_{h=1}^j x_{\ell h}.
\end{equation}
\end{enumerate}
\end{proposition}
A similar construction offers a correspondence between families of aggregation functions and alternating transportation polytopes with homogeneous marginals.  
Given two vectors $\tilde{u}:=(\tilde{u}_1, \ldots, \tilde{u}_p) \in \mathbb{R}^p_{>0}$ and $\tilde{v}:=(\tilde{v}_1, \ldots, \tilde{v}_q) \in \mathbb{R}^q_{>0}$, 
we define the set $\text{ASA}(\tilde{u},\tilde{v})$ to be the matrices $[c_{ij}]\in\R^{(p+1)\times(q+1)}$ which satisfy conditions (AF1a), (AF1b), and 
\begin{enumerate}
\item[(AF2b)] $c_{i_1\,j_1} + c_{i_2\,j_2} - c_{i_1\,j_2} - c_{i_2\,j_1} \geq 0$\;
for every \;$i_1 \leq i_2 \in \langle p \rangle$, $j_1 \leq j_2 \in \langle q \rangle$, \\[3pt]
and $i_1=0$, or $i_2=p$, or $j_1=0$, or $j_2=q$.
\end{enumerate}

It can be shown that the elements of $\text{ASA}(\tilde{u},\tilde{v})$ are discrete aggregation functions by following the same reasoning presented for the set $\text{SAF}(\tilde{u},\tilde{v})$. 
The following proposition shows the link between the set $\text{ASA}(\tilde{u},\tilde{v})$ and an alternating transportation polytope $\mathcal{A}(u,v)$.
%-PROPOSITION: Alternating Transportation Matrices and Aggregation Functions----
\begin{proposition}
\label{Prop:AggrFunAltPol}
For a function $C_{pq}:I_p \times I_q \rightarrow [0,1]$, the following statements are equivalent:
\begin{enumerate}
\item[(i)] $C_{pq} \in \text{ASA}(\tilde{u},\tilde{v})$.
\vspace{3pt}
\item[(ii)] There is a $(p \times q)$ alternating transportation matrix $[x_{ij}]$ in $\mathcal{A}(u,v)$, with $\sum_{h=1}^{q} v_h = \sum_{\ell=1}^{p} u_{\ell} = pq$, such that for every $i \in \langle p \rangle$, $j \in \langle q \rangle$ 
\begin{equation}\label{Eq:ASA_AT}
c_{ij}:=C_{p,q}\left(\frac{i}{p},\frac{j}{q}\right)= \frac{1}{pq} \sum_{\ell =1}^i \sum_{h=1}^j x_{\ell h}.
\end{equation}
\end{enumerate}
\end{proposition}
%

%---REMARK: On our generalization---
\begin{remark}
\label{rmk: on our generalization}
Propositions~\ref{Prop:AggrFuncTranPol} and~\ref{Prop:AggrFunAltPol} together offer a natural completion for the question marks in our above hierarchy on generalizations of discrete copulas that fits nicely within the current literature on copula functions.   
We note however, that the correspondence captured in these propositions does not capture all $p\times q$ (alternating) transportation polytopes, but only those with homogeneous marginals.  
For example, this generalized correspondence does not encompass the (alternating) transportation polytopes containing the polytopes considered in Theorem~\ref{thm: recursively constructed vertices}.  
To the best of the authors' knowledge there is no generalization of discrete copulas in the statistical literature that corresponds to the entire family of $p\times q$ alternating transportation polytopes.  
\end{remark}

%---------DISCUSSION-------
\section{Discussion}
\label{sec: discussion}
There has recently been an increasing interest in exploiting tools from the field of discrete geometry to develop new methodology in applied fields \citep{aghakouchak_2014,piantadosi_2007,piantadosi_2012,radi_2017} and shed light on well-known stochastic problems \citep{Krause_2017,embrechts2016,Fiebig2017}.
In this work, we unified the theoretical analysis of discrete copulas and their generalizations with the existing theory on generalizations of the Birkhoff polytope in the discrete geometry literature.  
Bivariate discrete copulas, and their generalizations discussed in this paper, admit representations as polytopes corresponding to generalizations of the Birkhoff polytope.  
We identified minimal $H$-representations of the families of $p\times q$ ultramodular bivariate discrete copulas and of $p\times q$ bivariate convex discrete quasi-copulas as subpolytopes of the $p\times q$ generalized Birkhoff polytope and the $p\times q$ generalized alternating sign matrix polytope, respectively.  
Along the way, we also generalized well-known results on alternating sign matrix polytopes by computing the minimal $H$-representation of the $p\times q$ generalized alternating sign matrix polytope.  
In addition, we presented new methods for constructing \emph{irreducible} elements of each of these families of $p\times p$ (quasi)-copulas by constructing families of vertices for the associated polytopes.  
Finally, we ended by connecting the most extensive generalization of discrete copulas in the statistical literature (i.e., aggregration functions) with the most extensive generalization of Birkhoff polytopes in the discrete geometry literature (i.e., alternating transportation polytopes), thereby completely unifying the two hierarchies of generalizations.
The geometric findings presented in this paper can be used to determine whether a given arbitrary nonnegative matrix is the probability mass of an ultramodular bivariate copula, thereby providing new tools for entropy-copula approaches in line with \cite{piantadosi_2007,piantadosi_2012}. Moreover, an interesting direction for future research is to build on our results to construct statistical tests of ultramodularity for bivariate copulas in the same fashion as symmetry tests \citep{Genest2012,Jasson_2005}. 
The extension results of Theorem~\ref{Th:ext} and Theorem~\ref{Th:extQC}, together with the vertex constructions presented in Section~\ref{sec: on vertex representations}, suggest alternative ways to obtain ultramodular bivariate copulas and convex quasi-copulas which could be used as smooth approximators of analytically unfeasible (quasi)-copulas.

A natural follow-up to this research is to define the geometry of multivariate discrete copulas with the property of ultramodularity. 
This would allow an efficient approximation of popular multivariate families of copulas such as Extreme Value \citep{Capera_1997,Gudendorf2010}, Archimedean \citep{Genest2011}, and Archimax \citep{Caperaa_2000,Charpentier2014}, which relate to the ultramodular ones \citep{Saminger-Platz_2017}.
Finally, it would be interesting to consider also other types of stochastic dependence for discrete copulas such as multivariate total positivity \citep{Colangelo_2006,Mueller_2015}.

%---ACKNOWLEDGEMENTS-------
\section*{Acknowledgments}
\label{sec: acknowledgements}
\noindent Elisa Perrone was supported by an FWF Schr\"odinger Fellowship (J 3968-N32).
Liam Solus was supported by an NSF Mathematical Sciences Postdoctoral Fellowship (DMS-1606407). Caroline Uhler was supported by NSF (DMS-1651995), ONR (N00014-17-1-2147), and a Sloan Fellowship.

\newpage
%---BIBLIOGRAPHY--------
%\section*{References}

%\bibliographystyle{abbrvnat}
%\bibliography{Final_2018}

\begin{thebibliography}{49}
\providecommand{\natexlab}[1]{#1}
\providecommand{\url}[1]{\texttt{#1}}
\expandafter\ifx\csname urlstyle\endcsname\relax
  \providecommand{\doi}[1]{doi: #1}\else
  \providecommand{\doi}{doi: \begingroup \urlstyle{rm}\Url}\fi

\bibitem[AghaKouchak(2014)]{aghakouchak_2014}
A.~AghaKouchak.
\newblock Entropyñ {Copula} in {Hydrology} and {Climatology}.
\newblock \emph{Journal of {Hydrometeorology}}, 15\penalty0 (6):\penalty0
  2176--2189, 2014.

\bibitem[Aguil\'{o} et~al.(2008)Aguil\'{o}, Su\~{n}er, and
  Torrens]{aguilo_2008}
I.~Aguil\'{o}, J.~Su\~{n}er, and J.~Torrens.
\newblock Matrix representation of discrete quasi-copulas.
\newblock \emph{Fuzzy Sets and Systems}, 159\penalty0 (13):\penalty0
  1658--1672, 2008.

\bibitem[Aguil\'{o} et~al.(2010)Aguil\'{o}, Su\~{n}er, and
  Torrens]{aguilo_2010}
I.~Aguil\'{o}, J.~Su\~{n}er, and J.~Torrens.
\newblock Matrix representation of copulas and quasi-copulas defined on
  non-square grids of the unit square.
\newblock \emph{Fuzzy Sets and Systems}, 161\penalty0 (2):\penalty0 254--268,
  2010.

\bibitem[Barvinok(2002)]{Barvinok_02}
A.~Barvinok.
\newblock \emph{A {C}ourse in {C}onvexity}, volume~54.
\newblock American Mathematical Society, Providence, RI, 2002.

\bibitem[Cap\'{e}ra\`{a} et~al.(1997)Cap\'{e}ra\`{a}, Foug\`{e}res, and
  Genest]{Capera_1997}
P.~Cap\'{e}ra\`{a}, A.-L. Foug\`{e}res, and C.~Genest.
\newblock A nonparametric estimation procedure for bivariate extreme value
  copulas.
\newblock \emph{Biometrika}, 84\penalty0 (3):\penalty0 567--577, 1997.

\bibitem[Cap\'{e}ra\`{a} et~al.(2000)Cap\'{e}ra\`{a}, Foug\`{e}res, and
  Genest]{Caperaa_2000}
P.~Cap\'{e}ra\`{a}, A.-L. Foug\`{e}res, and C.~Genest.
\newblock Bivariate distributions with given extreme value attractor.
\newblock \emph{Journal of Multivariate Analysis}, 72\penalty0 (1):\penalty0
  30--49, 2000.

\bibitem[Charpentier et~al.(2014)Charpentier, Foug\`{e}res, Genest, and
  Ne\v{s}lehov\'{a}]{Charpentier2014}
A.~Charpentier, A.-L. Foug\`{e}res, C.~Genest, and J.~Ne\v{s}lehov\'{a}.
\newblock Multivariate {A}rchimax copulas.
\newblock \emph{Journal of Multivariate Analysis}, 126\penalty0 (Supplement
  C):\penalty0 118--136, 2014.

\bibitem[Colangelo et~al.(2006)Colangelo, Scarsini, and Shaked]{Colangelo_2006}
A.~Colangelo, M.~Scarsini, and M.~Shaked.
\newblock Some positive dependence stochastic orders.
\newblock \emph{Journal of Multivariate Analysis}, 97\penalty0 (1):\penalty0
  46--78, 2006.

\bibitem[De~Baets and De~Meyer(2007)]{De_Baets_2007}
B.~De~Baets and H.~De~Meyer.
\newblock Ortogonal grid constructions of copulas.
\newblock \emph{IEEE Transactions on Fuzzy Systems}, 15\penalty0 (6):\penalty0
  1053--1062, 2007.

\bibitem[De~Loera and Kim(2014)]{DeLoera_14}
J.~A. De~Loera and E.~D. Kim.
\newblock Combinatorics and geometry of transportation polytopes: An update.
\newblock In \emph{Discrete Geometry and Algebraic Combinatorics}, volume 625
  of \emph{Contemporary Mathematics}, pages 37--76. American Mathematical
  Society, Providence, RI, 2014.

\bibitem[Durante and Sempi(2015)]{DurSem15}
F.~Durante and C.~Sempi.
\newblock \emph{Principles of Copula Theory}.
\newblock CRC/Chapman \& Hall, Boca Raton, FL, 2015.

\bibitem[Durante et~al.(2013)Durante, Fern\'{a}ndez~S\'{a}nchez, and
  Sempi]{Durante_13}
F.~Durante, J.~Fern\'{a}ndez~S\'{a}nchez, and C.~Sempi.
\newblock Multivariate patchwork copulas: A unified approach with applications
  to partial comonotonicity.
\newblock \emph{Insurance: Mathematics and Economics}, 53\penalty0
  (3):\penalty0 897--905, 2013.

\bibitem[Embrechts et~al.(2016)Embrechts, Hofert, and Wang]{embrechts2016}
P.~Embrechts, M.~Hofert, and R.~Wang.
\newblock Bernoulli and tail-dependence compatibility.
\newblock \emph{The Annals of Applied Probability}, 26\penalty0 (3):\penalty0
  1636--1658, 2016.

\bibitem[Fiebig et~al.(2017)Fiebig, Strokorb, and Schlather]{Fiebig2017}
U.-R. Fiebig, K.~Strokorb, and M.~Schlather.
\newblock The realization problem for tail correlation functions.
\newblock \emph{Extremes}, 20\penalty0 (1):\penalty0 121--168, 2017.

\bibitem[Gawrilow and Joswig(2000)]{polymake:2000}
E.~Gawrilow and M.~Joswig.
\newblock {\tt polymake}: a framework for analyzing convex polytopes.
\newblock In \emph{Polytopes---combinatorics and computation ({O}berwolfach,
  1997)}, volume~29 of \emph{DMV Seminars}, pages 43--73. 2000.

\bibitem[Genest et~al.(1999)Genest, Quesada~Molina, Rodr\'{i}guez~Lallena, and
  Sempi]{genest_characterization_1999}
C.~Genest, J.~J. Quesada~Molina, J.~A. Rodr\'{i}guez~Lallena, and C.~Sempi.
\newblock A {Characterization} of {Quasi}-copulas.
\newblock \emph{Journal of Multivariate Analysis}, 69\penalty0 (2):\penalty0
  193--205, 1999.

\bibitem[Genest et~al.(2011)Genest, Ne{\v{s}}lehov{\'a}, and
  Ziegel]{Genest2011}
C.~Genest, J.~Ne{\v{s}}lehov{\'a}, and J.~Ziegel.
\newblock Inference in multivariate {A}rchimedean copula models.
\newblock \emph{TEST}, 20\penalty0 (2):\penalty0 223--256, 2011.

\bibitem[Genest et~al.(2012)Genest, Ne{\v{s}}lehov{\'a}, and
  Quessy]{Genest2012}
C.~Genest, J.~Ne{\v{s}}lehov{\'a}, and J.-F. Quessy.
\newblock Tests of symmetry for bivariate copulas.
\newblock \emph{Annals of the Institute of Statistical Mathematics},
  64\penalty0 (4):\penalty0 811--834, 2012.

\bibitem[Genest et~al.(2014)Genest, Ne\v{s}lehov\'a, and
  R\'{e}millard]{genest_2014}
C.~Genest, J.~G. Ne\v{s}lehov\'a, and B.~R\'{e}millard.
\newblock On the empirical multilinear copula process for count data.
\newblock \emph{Bernoulli}, 20\penalty0 (3):\penalty0 1344--1371, 2014.

\bibitem[Genest et~al.(2017)Genest, Ne\v{s}lehov\'a, and
  R\'{e}millard]{genest_2017}
C.~Genest, J.~G. Ne\v{s}lehov\'a, and B.~R\'{e}millard.
\newblock Asymptotic behavior of the empirical multilinear copula process under
  broad conditions.
\newblock \emph{Journal of Multivariate Analysis}, 159:\penalty0 82 -- 110,
  2017.

\bibitem[Grabisch et~al.(2009)Grabisch, Marichal, Mesiar, and
  Pap]{Grabisch_2009}
M.~Grabisch, J.-L. Marichal, R.~Mesiar, and E.~Pap.
\newblock \emph{Aggregation Functions (Encyclopedia of Mathematics and Its
  Applications)}.
\newblock Cambridge University Press, 1st edition, 2009.

\bibitem[Gudendorf and Segers(2010)]{Gudendorf2010}
G.~Gudendorf and J.~Segers.
\newblock Extreme-value copulas.
\newblock In P.~Jaworski, F.~Durante, W.~K. H{\"a}rdle, and T.~Rychlik,
  editors, \emph{Copula Theory and Its Applications: Proceedings of the
  Workshop Held in Warsaw, 25-26 September 2009}, pages 127--145. Springer
  Berlin Heidelberg, 2010.

\bibitem[Jasson(2005)]{Jasson_2005}
S.~Jasson.
\newblock L'asym\'{e}trie de la d\'{e}pendance, quel impact sur la
  tarification?
\newblock \emph{Technical report, AXA Group Risk Management, Paris, France},
  2005.

\bibitem[Joe(2014)]{joe_2014}
H.~Joe.
\newblock \emph{{Dependence Modeling with Copulas}}.
\newblock Chapman and Hall/CRC, Boca Raton, FL, 2nd edition, 2014.

\bibitem[Klement et~al.(2011)Klement, Manzi, and Mesiar]{klement_2011}
E.~P. Klement, M.~Manzi, and R.~Mesiar.
\newblock Ultramodular aggregation functions.
\newblock \emph{Information Sciences}, 181\penalty0 (19):\penalty0 4101--4111,
  2011.

\bibitem[Klement et~al.(2014)Klement, Manzi, and Mesiar]{klement_2014}
E.~P. Klement, M.~Manzi, and R.~Mesiar.
\newblock Ultramodularity and copulas.
\newblock \emph{Rocky Mountain Journal of Mathematics}, 44\penalty0
  (1):\penalty0 189--202, 2014.

\bibitem[Klement et~al.(2017)Klement, Koles{\'a}rov{\'a}, Mesiar, and
  Saminger-Platz]{klement_2017b}
E.~P. Klement, A.~Koles{\'a}rov{\'a}, R.~Mesiar, and S.~Saminger-Platz.
\newblock Copula constructions using ultramodularity.
\newblock In M.~{\'U}beda~Flores, E.~de~Amo~Artero, F.~Durante, and
  J.~Fern{\'a}ndez~S{\'a}nchez, editors, \emph{Copulas and Dependence Models
  with Applications}, pages 135--156, Cham, 2017. Springer International
  Publishing.

\bibitem[Knight(2009)]{Knight_2009}
V.~Knight.
\newblock \emph{Alternating Sign Matrices and Polytopes}.
\newblock PhD thesis, 2009.

\bibitem[Koles\'arov\'a et~al.(2006)Koles\'arov\'a, Mesiar, Mordelov\'a, and
  Sempi]{kolesarova_2006}
A.~Koles\'arov\'a, R.~Mesiar, J.~Mordelov\'a, and C.~Sempi.
\newblock Discrete {Copulas}.
\newblock \emph{IEEE Transactions on Fuzzy Systems}, 14\penalty0 (5):\penalty0
  698--705, 2006.

\bibitem[Krause et~al.(2017)Krause, Scherer, Schwinn, and Werner]{Krause_2017}
D.~Krause, M.~Scherer, J.~Schwinn, and R.~Werner.
\newblock Membership testing for bernoulli and tail-dependence matrices.
\newblock \emph{Working Paper}, 2017.

\bibitem[Marinacci and Montrucchio(2005)]{marinacci_2005}
M.~Marinacci and L.~Montrucchio.
\newblock Ultramodular {Functions}.
\newblock \emph{Mathematics of Operations Research}, 30\penalty0 (2):\penalty0
  311--332, 2005.

\bibitem[Mayor et~al.(2005)Mayor, Su\~{n}er, and Torrens]{Mayor_2005}
G.~Mayor, J.~Su\~{n}er, and J.~Torrens.
\newblock Copula-like operations on finite settings.
\newblock \emph{IEEE Transanctions on Fuzzy Systems}, 13\penalty0 (4):\penalty0
  468--477, 2005.

\bibitem[Mesiar(2005)]{Mesiar_2005}
R.~Mesiar.
\newblock Discrete copulas-what they are.
\newblock In \emph{Proceedings of {EUSFLAT- LFA} Conference (Barcelona,
  Spain)}, pages 927--930. 2005.

\bibitem[Mesiar and Szolgay(2004)]{Mesiar_2004}
R.~Mesiar and J.~Szolgay.
\newblock W-ordinals sum of copulas and quasi-copulas.
\newblock In \emph{Proceedings of MAGIA 2004 Conference (Koc\v{o}vce, Slovak
  Republic)}, pages 78--83. 2004.

\bibitem[Mordelov\'a and Koles\'arov\'a(2007)]{mordelova_2017}
J.~Mordelov\'a and A.~Koles\'arov\'a.
\newblock Some results on discrete copulas.
\newblock In \emph{Proceedings of the Fourth International Summer School on
  Aggregation operatos (Ghent, Belgium)}, pages 145--150, 2007.

\bibitem[M\"{u}ller and Scarsini(2001)]{mueller_2001}
A.~M\"{u}ller and M.~Scarsini.
\newblock Stochastic comparison of random vectors with a common copula.
\newblock \emph{Mathematics of Operations Research}, 26\penalty0 (4):\penalty0
  723--740, 2001.

\bibitem[M\"{u}ller and Scarsini(2005)]{Mueller_2015}
A.~M\"{u}ller and M.~Scarsini.
\newblock Archimedean copulae and positive dependence.
\newblock \emph{Journal of Multivariate Analysis}, 93\penalty0 (2):\penalty0
  434--445, 2005.

\bibitem[M\"{u}ller and Scarsini(2012)]{Mueller_2011}
A.~M\"{u}ller and M.~Scarsini.
\newblock Fear of loss, inframodularity, and transfers.
\newblock \emph{Journal of Economic Theory}, 147\penalty0 (4):\penalty0
  1490--1500, 2012.
\newblock Inequality and Risk.

\bibitem[Nelsen(2006)]{nelsen_06}
R.~B. Nelsen.
\newblock \emph{{An Introduction to Copulas (Springer Series in Statistics)}}.
\newblock Springer, 2nd edition, 2006.

\bibitem[Nelsen and Flores(2005)]{Nelsen_2005}
R.~B. Nelsen and M.~U. Flores.
\newblock The lattice-theoretic structure of sets of bivariate copulas and
  quasi-copulas.
\newblock \emph{Comptes Rendus Mathematique}, 341\penalty0 (9):\penalty0
  583--586, 2005.

\bibitem[Piantadosi et~al.(2007)Piantadosi, Howlett, and
  Boland]{piantadosi_2007}
J.~Piantadosi, P.~Howlett, and J.~Boland.
\newblock Matching the grade correlation coefficient using a copula with
  maximum disorder.
\newblock \emph{Journal of Industrial \& Management Optimization}, 3\penalty0
  (1):\penalty0 305--312, 2007.

\bibitem[Piantadosi et~al.(2012)Piantadosi, Howlett, and
  Borwein]{piantadosi_2012}
J.~Piantadosi, P.~Howlett, and J.~Borwein.
\newblock Copulas with maximum entropy.
\newblock \emph{Optimization Letters}, 6\penalty0 (1):\penalty0 99--125, 2012.

\bibitem[Quesada~Molina and Sempi(2005)]{quesada_molina_discrete_2005}
J.~J. Quesada~Molina and C.~Sempi.
\newblock Discrete quasi-copulas.
\newblock \emph{Insurance: Mathematics and Economics}, 37\penalty0
  (1):\penalty0 27--41, 2005.
\newblock ISSN 0167-6687.

\bibitem[Radi et~al.(2017)Radi, Zakaria, Piantadosi, Boland, Zin, and
  Azman]{radi_2017}
N.~F.~A. Radi, R.~Zakaria, J.~Piantadosi, J.~Boland, W.~Z.~W. Zin, and M.~A.-z.
  Azman.
\newblock Generating {Synthetic} {Rainfall} {Total} {Using} {Multivariate}
  {Skew}-\emph{t} and
  {Checkerboard} {Copula} of {Maximum} {Entropy}.
\newblock \emph{{Water} {Resources} {Management}}, 31\penalty0 (5):\penalty0
  1729--1744, 2017.

\bibitem[Saminger-Platz et~al.(2017)Saminger-Platz, Klement, De~Jes\'{u}s
  Arias-Garc\'{i}a, and Mesiar]{Saminger-Platz_2017}
S.~Saminger-Platz, E.~P. Klement, J.~De~Jes\'{u}s Arias-Garc\'{i}a, and
  R.~Mesiar.
\newblock Characterizations of bivariate conic, extreme value, and {A}rchimax
  copulas.
\newblock \emph{Dependence Modeling}, 5\penalty0 (1):\penalty0 45--58, 2017.

\bibitem[Scaillet et~al.(2007)Scaillet, Charpentier, and
  Fermanian]{Scaillet2007}
O.~Scaillet, A.~Charpentier, and J.-D. Fermanian.
\newblock \emph{The estimation of copulas : theory and practice}, pages 35--64.
\newblock Copulas: from theory to application in finance. Risk Books, London,
  2007.

\bibitem[Sklar(1959)]{sklar_59}
A.~Sklar.
\newblock {Fonctions de r\'epartition \`a n dimensions et leurs marges}.
\newblock \emph{Publications de l'Institut de Statistique de Paris},
  8:\penalty0 229--231, 1959.

\bibitem[Striker(2009)]{striker_2009}
J.~Striker.
\newblock The {Alternating} {Sign} {Matrix} {Polytope}.
\newblock \emph{The Electronic Journal of Combinatorics}, 16\penalty0
  (1):\penalty0 R41, 2009.

\bibitem[Ziegler(1995)]{Ziegler_95}
G.~M. Ziegler.
\newblock \emph{Lectures on Polytopes}.
\newblock Springer-Verlag, 1995.

\end{thebibliography}

%---APPENDIX-----------
\section*{Appendix A: Proofs for Section~\ref{sec: the polytope of ultramodular discrete copulas}}

%---LEMMA: Containment (Lemma 1)----
\begin{lemma}
\label{Lem:1}
Suppose that $[c_{ij}]\in\R^{(p+1)\times(q+1)}$ satisfies all of (d1), (d2), (d3a), and (d3b) as well as the equalities
\[
c_{0k}=0, \, c_{pk}=\frac{k}{q}, \,  c_{h0}=0, \, c_{hq}=\frac{h}{p} \, \mbox{for all } h \in \langle p\rangle, k \in \langle q\rangle.
\]  
Then $[c_{ij}]\in\UDC_{p,q}$.  
\end{lemma}

%%---SUBSECTION: Proof of Lemma~\ref{Lem:1}
%\subsection{Proof of Lemma~\ref{Lem:1}.}
%%
\begin{proof}
To prove the result, we consider $C = [c_{ij}]\in\R^{(p+1)\times(q+1)}$ that satisfies all of the inequalities (d1), (d2), (d3a), and (d3b) together with the equalities stated in the lemma. 
To show $C\in\UDC_{p,q}$, we must check that $C$ satisfies the inequalities (c1).  
That is, we must show that the following inequalities are valid on $C$.  
\[
   \begin{array}{ll}
      \textbf{i.} & c_{1 1} + c_{2 2} - c_{1 2} - c_{2 1} \geq 0; \\[6pt]
      \textbf{ii.} & c_{p-2,q-2} + c_{p-1,q-1} - c_{p-2,q-1} - c_{p-1,q-2} \geq 0; \\[6pt]
      \textbf{iii.} & (a) \; c_{1,j+1} - c_{1 j} \geq 0, \; (b) \; c_{i+1, 1} - c_{i 1} \geq 0; \quad \\
      			& \mbox{for all } i \in \langle p-1\rangle, \, j \in \langle q-1\rangle \\[6pt]
      \textbf{iv.} & (a) \; c_{p-1,j+1} - c_{p-1,j} \leq \tfrac{1}{q}, \; (b) \; c_{i+1,q-1} - c_{i,q-1} \leq \tfrac{1}{p}; \quad \\
      			& \mbox{for all } i \in \langle p-1\rangle, \, j \in \langle q-1\rangle \\[6pt]
   \end{array} 
\]

$C$ satisfies conditions (d3a) and (d3b), respectively for $(i,j)=(2,0)$ and $(i,j)=(0,2)$. Moreover, $c_{11}\geq 0$. Therefore, inequality \textbf{i.} can be obtained from
\[
2(c_{11} + c_{22}) \geq 2 c_{22} \geq 2(c_{12} + c_{21}).
\]

From inequalities (d1), it holds that $c_{p-1,q-1}\geq -\tfrac{1}{p}+\tfrac{q-1}{q}$. Assuming $q\geq p$, we then have that
\[
c_{p-1,q-1}\geq \max \{-\tfrac{1}{q}+\tfrac{q-1}{q},-\tfrac{1}{p}+\tfrac{p-1}{p}\}= \max \{ \tfrac{q-2}{q},\tfrac{p-2}{p} \}.
\]

From (d3a) and (d3b) for $(i,j)=(p-2,q-2)$, we recover inequality \textbf{ii}:
\[
2 \; \Big( c_{p-2, q-2}  + c_{p-1, q-1}  \Big) \geq 
2 \; c_{p-2, q-2} + \tfrac{p-2}{p} + \tfrac{q-2}{q} \geq 2 \; \Big( c_{p-2, q-1}  + c_{p-1, q-2}  \Big).
\]
The inequalities \textbf{iii.(a)} and \textbf{iv.(a)} can be obtained by combining conditions (d1) and (d3a).  Indeed, for \textbf{iii.(a)} we have that 
\[
 c_{1,j+2}  - c_{1,j+1} \geq^{{\text{(d3a)}}} \; c_{1,j+1}  - c_{1 j} \geq \ldots \geq  \; c_{1 2}  + c_{1 1} \geq \; c_{1 1}   \geq^{{\text{(d1)}}} \; 0.
\]
Similarly, for \textbf{iv.(a)} we have that
\[
 \tfrac{1}{q} \geq^{{\text{(d1)}}} \; c_{p-1,q} - c_{q-1,q-1} \geq^{{\text{(d3a)}}} \; c_{p-1,j+2}  - c_{p-1,j+1}\geq \; c_{p-1,j+1} - c_{p-1,j}. 
\]
In an analogous manner one can derive \textbf{iii.(b)} and \textbf{iv.(b)}.
\end{proof}

%---SUBSECTION: Proof of Theorem~\ref{Th:h-repUC}----
\subsection{Proof of Theorem \ref{Th:h-repUC}.}
We here prove that the inequalities in the list (b1), (b2), (b3a), and (b3b) are the minimal H-representation of the polytope $\TUDC_{p,q}$. 
To do this, we identify $(p \times q)$-matrices $M^{(ij)}_{pq}=[b_{ij}]$, and $H^{(ij)}_{pq}=[h_{ij}]$ for $i \in [p]$ and $j \in [q]$ such that
\begin{itemize}
\item[]{\textsc{\bf Case }(b1).} for every $p$ and $q$, $M^{(11)}_{pq}$ satisfies all inequalities in the list (b1), (b2), (b3a), and (b3b) except for inequality of the type $b_{11}\geq 0$.
\item[]{\textsc{\bf Case }(b2).} for every $i=2,\ldots,p-1$ and $j=2, \ldots, q-1$, except for $(i,j)=\{(2,2),(p-1,q-1)\}$, $M^{(ij)}_{pq}$ satisfies all inequalities in the list (b1), (b2), (b3a), and (b3b) but one of the type $b_{ij}\geq 0$. 
\item[]{\textsc{\bf Case }(b3a).} for every  $i \in [p-1]$ and $1 \leq j \leq \floor*{\tfrac{q+1}{2}}$, $H^{(ij)}_{pq}$ satisfies all inequalities in the list (b1), (b2), (b3a), and (b3b) except for one of the type $ \sum\limits_{h=1}^j b_{i+1, h} \geq \sum\limits_{h=1}^j b_{ih}$. 
\end{itemize}
The matrices that we shall identify satisfying each of these cases are, collectively, sufficient to prove that every inequality in the list (b1), (b2), (b3a), and (b3b) is needed to bound the polytope $\TUDC_{p,q}$. 
Indeed, let us assume $M^{(ij)}_{pq}=(b_{ij})$ to be a matrix that satisfies (b1), (b2), (b3a), and (b3b), but for $b_{\hat{i}\hat{j}}\geq 0$ with $\hat{i} \in \{2,\ldots, \floor*{\tfrac{p+1}{2} }\}$ and $\hat{j} \in \{2,\ldots, \floor*{\tfrac{q+1}{2}} \}$.
Then the matrix $M^{(p-\hat{i}+1,q-\hat{j}+1)}_{pq}=(b_{p-\hat{i}+1, q-\hat{j}+1})$ obtained by flipping the original matrix $M^{(ij)}_{pq}=(b_{ij})$ as follows
\[
 M^{(p-\hat{i}+1,q-\hat{j}+1)}_{pq}= \begin{pmatrix}
	\; b_{p q} \; & \; b_{p,q-1} \; & \ldots & \; b_{p 1} \; \\
	 b_{p-1,q}  &  b_{p-1,q-1}  & \ldots &  b_{p-1,1} \\
	\vdots & & & \vdots \\
	\; b_{1 q}\; &\; b_{1,q-1} \;& \ldots & \;b_{1 1}\; \\
	\end{pmatrix}
\]
satisfies all of the constraints but for $b_{p-\hat{i}+1, q-\hat{j}+1}\geq 0$. We indicate this transformation with $b_{ij}^F$.
In an analogous manner, one can obtain all of the remaining cases among inequalities (b3a).
 Moreover, matrices that satisfy all the inequalities of $\TUDC_{p,q}$ except for one of the $(b3b)$-type can be obtained by transposing the ones of case (b3a) above.
 
We now present the matrices corresponding to cases (b1), (b2), and (b3a) listed above.  
The inequalities considered in each of these three cases are further subdivided into the following subcases.
Following the list of subcases for each case, we present the matrices satisfying all inequalities in the list (b1), (b2), (b3a), and (b3b) with the exception of the inequality corresponding to the given subcase.      
When considered together with Lemma~\ref{Lem:1}, these subcases and their corresponding matrices complete the proof.  

The subcases of case (b1) are the following:
%---LIST OF SUBCASES FOR CASE (b1) ----
\begin{itemize}
	\item[{\bf A.}] ${b_{11} < 0} \text{ with } {p,q\geq4}.$
\end{itemize}
The following is the associated list of matrices for the subcases of case (b1) listed above.
%---LIST OF MATRICES FOR CASE (b1)----
\begin{itemize}
	%---SUBCASE A-----
	\item[{\bf A.}] 
	{$\begin{pmatrix}
	& \hspace{-0.5cm}  {-1} & {1} & \; \textcolor{black}{1} \;& { 2} & { 2} \\[3pt]
	 & \hspace{-0.5cm} { 1} & {0} & \; {1} \;& {1} & {2}\\[3pt]
   	 \begin{smallmatrix} {p - 4\Big[} \end{smallmatrix} &  \hspace{-0.5cm} {1} & {1} & \; {1} \;& {1} & {1}\\[3pt]
   	 & \hspace{-0.5cm} {2}  & {1} & \;{1} \; & {1} & { 0}\\[3pt]
    & \hspace{-0.5cm} {2}  & {2} & {\smash{\underbrace{\;{1} \;}_{q-4} }} & {0} & {0}\\[3pt]
\end{pmatrix}$}.
\end{itemize}

The subcases of case (b2) are the following:
%---LIST OF SUBCASES FOR CASE (b2)------
\begin{itemize}
	%---SUBCASE B1-----
	\item[{\bf B1.}] ${b_{23} < 0} \text{ with } {p=4} \text{ and } q=5;$
	%---SUBCASE B2-----
	\item[{\bf B2.}] ${b_{2j} < 0} \text{ with }{3 \leq j \leq \floor*{\tfrac{q}{2}}}, \text{ for } {p \geq 4}, {q \geq 6};$
	%---SUBCASE B3-----
	\item[{\bf B3.}] ${b_{2 \floor*{\tfrac{q+1}{2}} } < 0}, \text{ with } {p \geq 4}, {q \geq 7}, \text{ $q$ odd};$
	%---SUBCASE B4-----
	\item[{\bf B4.}] ${b_{32} < 0}, \text{ with } {p =5}, {q \geq 4};$
	%---SUBCASE B5-----
	\item[{\bf B5.}] ${b_{3j} < 0}, \text{ with } {3 \leq j \leq \floor*{\tfrac{q}{2}}}, \text{ and } {p=4, q \geq 4};$
	%---SUBCASE B6-----
	\item[{\bf B6.}] ${b_{3j} < 0}, \text{ with } 3 \leq j \leq \floor*{\tfrac{q}{2}}, \text{ and } {p =5}, {q \geq 6};$
	%---SUBCASE B7-----
	\item[{\bf B7.}] ${b_{3\floor*{\tfrac{q+1}{2}}} < 0}, \text{ with } {p =5}, {q \geq 5} \text{ $q$ odd};$ 
	%---SUBCASE B8-----
	\item[{\bf B8.}] ${b_{ij} < 0}, \text{ with } 3 \leq i \leq \floor*{\tfrac{p}{2}}, \, 3 \leq j \leq \floor*{\tfrac{q}{2}}, \text{ and } {p \geq 6}, \, q \geq 6 ;$
	%---SUBCASE B9-----
	\item[{\bf B9.}] ${b_{i\floor*{\tfrac{q+1}{2}}} < 0}, \text{ with } 3 \leq i \leq \floor*{\tfrac{p}{2}}, \text{ and } {p \geq 6}, {q \geq 7}, \text{ $q$ odd};$
	%---SUBCASE B10-----
	\item[{\bf B10.}] ${b_{\floor*{\tfrac{p+1}{2}}j} < 0}, \text{ with } 3 \leq j \leq \floor*{\tfrac{q}{2}}, \text{ and } {p \geq 7}, {q \geq 6}, \text{ $p$ odd};$
	%---SUBCASE B11-----
	\item[{\bf B11.}] ${b_{\floor*{\tfrac{p+1}{2}}\floor*{\tfrac{q+1}{2}}} < 0}, \text{ with } {p, q \geq 7}, \text{ $p$ and $q$ odd};$
	%---SUBCASE B12-----
	\item[{\bf B12.}] ${b_{4 2} < 0}, \text{ with } {p=5}, q \geq 5;$
	%---SUBCASE B13-----
	\item[{\bf B13.}] ${b_{p-1,2} < 0}, \text{ with } {p \geq 6}, {q \geq 6};$
	%---SUBCASE B14-----
	\item[{\bf B14.}] ${b_{p-1, j} < 0}, \text{ with } 3 \leq j \leq \floor*{\tfrac{q}{2}}, \, {p \geq 5}, {q \geq 6}, \text{ and $p$ odd};$
	%---SUBCASE B15-----
	\item[{\bf B15.}] ${b_{p-i+1, j} < 0}, \text{ with } 2 \leq i \leq \floor*{\tfrac{p}{2}}, \, 2 \leq j \leq \floor*{\tfrac{q}{2}}, \, {p \geq 6}, {q \geq 6}, \text{ $p$ odd};$
\end{itemize}
The following is the associated list of matrices for the subcases of case (b2) listed above.
%---LIST OF MATRICES FOR CASE (b2)----
\begin{itemize}
	%---SUBCASE B1-----
	\item[{\bf B1.}]
	{$\begin{pmatrix}
0 & 0 & 1 & 2 & 2 \\[3pt]
\tfrac{2}{3} & \tfrac{2}{3} & {-\tfrac{1}{3}} & 2 & 2 \\[3pt]
\tfrac{4}{3} & \tfrac{4}{3} & \tfrac{7}{3} & 0 & 0 \\[3pt]
2 & 2 & 1 & 0 & 0 \\[3pt]
\end{pmatrix}$}\\[3pt]

	%---SUBCASE B2-----
	\item[{\bf B2.}]
	{$\begin{pmatrix}
		& \hspace{-0.5cm}{0} & {0} & { 0} & {1} &{1} & {1} &{2} \\[3pt]
		 & \hspace{-0.5cm}{0} & {\tfrac{2}{3}} & {\tfrac{2}{3}} & {-\tfrac{1}{3}} & {1} & {1} & {2} \\[3pt]
		\begin{smallmatrix} {p - 4\Big[ }\end{smallmatrix} & \hspace{-0.5cm} {1} & {1} & {1} & {1} & {1} & {1} & {1}  \\[3pt]
		& \hspace{-0.5cm}{2} & {\tfrac{4}{3}} & {\tfrac{4}{3}} & {\tfrac{4}{3}} & {1} & {2} & {0} \\[3pt]
		& \hspace{-0.5cm}{\smash{\underbrace{\;{2} \;}_{j-3}} }& {2} & {2} & {2} & {\smash{\underbrace{\;{1} \;}_{q-2j}}} & { 0} &{\smash{\underbrace{\; {0} \;}_{j-1}}} \\[3pt]
\end{pmatrix}$}\\[3pt]

	%---SUBCASE B3-----
	\item[{\bf B3.}]
	{$\begin{pmatrix}
		& \hspace{-0.5cm}{0} & { 0} & {0} & {1} & {2} & {2} & {2} \\[3pt]

		 & \hspace{-0.5cm}{0} & {\tfrac{2}{3}} & {\tfrac{2}{3}} & {-\tfrac{1}{3}} & {2} & { 2} &{2} \\[3pt]

		\begin{smallmatrix} {p - 4\Big[ }\end{smallmatrix} & \hspace{-0.5cm} {1} & {1} & {1} & {1} & {1} &{1} & {1} \\[3pt]

		& \hspace{-0.5cm}{2} & {\tfrac{4}{3}} & {\tfrac{4}{3}} & { \tfrac{7}{3}} &  {0} & { 0} & {0} \\[3pt]
		
		& \hspace{-0.5cm}{\smash{\underbrace{\;{2} \;}_{j-3}}} & { 2} & {2} & {1} & {0} & { 0} & {\smash{\underbrace{\; {0} \;}_{j-3}}} \\[3pt]
\end{pmatrix}$}\\[3pt]

	%---SUBCASE B4-----
	\item[{\bf B4.}]
	{$\begin{pmatrix}
0 & 1 & {1} & 1 & 2 \\[3pt]
0 & 1 & {1} & 1 & 2 \\[3pt]
\tfrac{3}{2} & {-\tfrac{1}{2}} & {1} & 2 & 1 \\[3pt]
\tfrac{3}{2} & \tfrac{3}{2} & {1} & 1 & 0 \\[3pt]
{2} & {2} & {\smash{\underbrace{\; {1} \;}_{q-4}}} & 0 & 0 \\[3pt]
\end{pmatrix}$}\\[3pt]

	%---SUBCASE B5-----
	\item[{\bf B5.}]
	{$\begin{pmatrix}
\textcolor{black}{0} & 1 & \textcolor{black}{1} & 1 & 2  &\textcolor{black}{2}\\[3pt]
\textcolor{black}{0} & 1+ y & \textcolor{black}{1 + y} & 1+y & 1  &\textcolor{black}{2}\\[3pt]
\textcolor{black}{2} & {-y} & \textcolor{black}{x} & 1+x & 1  &\textcolor{black}{0}\\[3pt]
\textcolor{black}{\smash{\underbrace{\; \textcolor{black}{2} \;}_{j-1}}} & 2 &  \textcolor{black}{\smash{\underbrace{\; \textcolor{black}{1} \;}_{q-2j}} } &   0 & 0 & \textcolor{black}{\smash{\underbrace{\; \textcolor{black}{0} \;}_{j-2}} } \\[6pt]
\end{pmatrix}$}
where
\[
y= \frac{1}{q-2j + 2}
\qquad \mbox{and} 
\qquad
x= \frac{q-2j+1}{q-2j + 2}.
\]

	%---SUBCASE B6-----
	\item[{\bf B6.}]
	{$\begin{pmatrix}
{0} & 0 & 1 &{1} & 1 & {2} \\[3pt]
{0} & 0 & 1 &{1} & 1 & {2} \\[3pt]
{1} & \tfrac{3}{2}& {-\tfrac{1}{2}} &{1} & 2 & {1} \\[3pt]
{2} & \tfrac{3}{2} & \tfrac{3}{2} & {1} & 1 & {0} \\[3pt]
{\smash{\underbrace{\; {2} \;}_{j-2}}} & 2 & 2 & {\smash{\underbrace{\; {1} \;}_{q-2j}}} & 0 & {\smash{\underbrace{\; {0} \;}_{j-1}}} \\[3pt]
\end{pmatrix}$}\\[3pt]

	%---SUBCASE B7-----
	\item[{\bf B7.}]
	{$\begin{pmatrix}
0 & {0} & 0 & 1 & 2 & {2} \\[3pt]
0 &{0} & 0 & 2 & 1 & {2} \\[3pt]
1 & {1} & 2 & {-1} & 2 & {1} \\[3pt]
1 & {2} & 2 & 2 & 0 & {0} \\[3pt]
3 & {\smash{\underbrace{\; {2} \;}_{j-3}}} & 1 & 1 & 0 & {\smash{\underbrace{\; {0} \;}_{j-2}}} \\[3pt]
\end{pmatrix}$}\\[3pt]

	%---SUBCASE B8-----
	\item[{\bf B8.}]
	{$\begin{pmatrix}
		\begin{smallmatrix} \textcolor{black}{i-3\Big[} \end{smallmatrix}& \hspace{-0.5cm}\textcolor{black}{0} & \textcolor{black}{0} & \textcolor{black}{1} &  \textcolor{black}{2} & \textcolor{black}{2} & \textcolor{black}{2} \\[6pt]

		 & \hspace{-0.5cm}\textcolor{black}{0} & {  \tfrac{2}{3}} & \textcolor{black}{1} & {  \tfrac{4}{3}} &  \textcolor{black}{2} & {  2}  \\[6pt]

		 & \hspace{-0.5cm}\textcolor{black}{0} & {  \tfrac{2}{3}} & \textcolor{black}{1} & {  \tfrac{4}{3}} &  \textcolor{black}{2} & {  2}  \\[6pt]

		 & \hspace{-0.5cm}\textcolor{black}{1} &  {-\tfrac{1}{3}} & \textcolor{black}{1} & {  \tfrac{4}{3}} &  \textcolor{black}{1} & {  2}  \\[6pt]
		
		\begin{smallmatrix} \textcolor{black}{p - 2i\Big[} \end{smallmatrix} & \hspace{-0.5cm} \textcolor{black}{1} & \textcolor{black}{1} & \textcolor{black}{1} &  \textcolor{black}{1} & \textcolor{black}{1} & \textcolor{black}{1}  \\[3pt]

		& \hspace{-0.5cm}\textcolor{black}{1} & {  1} & \textcolor{black}{1} & {  2} &  \textcolor{black}{1} & {  0} \\[6pt]
		
		\begin{smallmatrix} \textcolor{black}{i-1\Big[} \end{smallmatrix}& \hspace{-0.5cm}\textcolor{black}{\smash{\underbrace{\; \textcolor{black}{2} \;}_{j-1}}} & \textcolor{black}{2} & \textcolor{black}{\smash{\underbrace{\; \textcolor{black}{1} \;}_{q-2j}}} &  \textcolor{black}{0} &  \textcolor{black}{\smash{\underbrace{\; \textcolor{black}{0} \;}_{j-2}}} &  \textcolor{black}{0}\\[6pt]
\end{pmatrix}$}\\[3pt]

	%---SUBCASE B9-----
	\item[{\bf B9.}]
	{$\begin{pmatrix}

		\begin{smallmatrix} \textcolor{black}{i-1\Big[} \end{smallmatrix}& \hspace{-0.5cm}\textcolor{black}{0} & \textcolor{black}{0} & \textcolor{black}{0} &  \textcolor{black}{1} & \textcolor{black}{2}  \\[6pt]

		 & \hspace{-0.5cm}\textcolor{black}{0} & {  \tfrac{2}{3}} & {  \tfrac{2}{3}} & {-\tfrac{1}{3}}  &  \textcolor{black}{2}  \\[6pt]
		
		\begin{smallmatrix} \textcolor{black}{p - 2i\Big[} \end{smallmatrix} & \hspace{-0.5cm} \textcolor{black}{1} & \textcolor{black}{1} & \textcolor{black}{1} &  \textcolor{black}{1} & \textcolor{black}{1}   \\[3pt]

		& \hspace{-0.5cm}\textcolor{black}{2} & {  \tfrac{4}{3}} & {  \tfrac{4}{3}} & {  \tfrac{7}{3}} &  \textcolor{black}{0} \\[6pt]
		
		\begin{smallmatrix} \textcolor{black}{i-1\Big[} \end{smallmatrix}& \hspace{-0.5cm}\textcolor{black}{\smash{\underbrace{\; \textcolor{black}{2} \;}_{j-3}}} & \textcolor{black}{2} & \textcolor{black}{2}  &  \textcolor{black}{1} &  \textcolor{black}{\smash{\underbrace{\; \textcolor{black}{0} \;}_{j-1}}}\\[6pt]
				
\end{pmatrix}
$}\\[3pt]

	%---SUBCASE B10-----
	\item[{\bf B10.}]
	{$\begin{pmatrix}
		\begin{smallmatrix} \textcolor{black}{i-3\Big[} \end{smallmatrix}& \hspace{-0.5cm}\textcolor{black}{0} & \textcolor{black}{0} & \textcolor{black}{1} &  \textcolor{black}{2} & \textcolor{black}{2} \\[6pt]

		 & \hspace{-0.5cm}\textcolor{black}{0} & {  \tfrac{2}{3}} & \textcolor{black}{1} & {  \tfrac{4}{3}} &  \textcolor{black}{2}  \\[6pt]

		 & \hspace{-0.5cm}\textcolor{black}{0} & {  \tfrac{2}{3}} & \textcolor{black}{1} & {  \tfrac{4}{3}} &  \textcolor{black}{2}  \\[6pt]

		 & \hspace{-0.5cm}\textcolor{black}{1} &  {-\tfrac{1}{3}} & \textcolor{black}{1} & {  \tfrac{7}{3}} &  \textcolor{black}{1}  \\[6pt]

		\begin{smallmatrix} \textcolor{black}{i-1\Big[} \end{smallmatrix}& \hspace{-0.45cm}\textcolor{black}{\smash{\underbrace{\; \textcolor{black}{2} \;}_{j-1}} }& \textcolor{black}{2} & \textcolor{black}{\smash{\underbrace{\; \textcolor{black}{1} \;}_{q-2j}}} &  \textcolor{black}{0} &  \textcolor{black}{\smash{\underbrace{\; \textcolor{black}{0} \;}_{j-1}} }\\[6pt]
				
\end{pmatrix}$}\\[3pt]

	%---SUBCASE B11-----
	\item[{\bf B11.}]
	{$\begin{pmatrix}
		\begin{smallmatrix} \textcolor{black}{i-2\Big[} \end{smallmatrix}& \hspace{-0.5cm}\textcolor{black}{0} & \textcolor{black}{0} & \textcolor{black}{0} &  \textcolor{black}{1} & \textcolor{black}{2} & \textcolor{black}{2}  & \textcolor{black}{2} \\[6pt]

		 & \hspace{-0.5cm}\textcolor{black}{0} & {  1} & {  1} & {  1} &  {  1}  & {  1} &\textcolor{black}{2}  \\[6pt]

		  & \hspace{-0.5cm}\textcolor{black}{1} & {  1} & {  1} & {-1} &  {  2}  & {  2} &\textcolor{black}{1}  \\[6pt]

		 & \hspace{-0.5cm}\textcolor{black}{2} & {  1} & {  1} & {  3} &  {  0} & {  0} & \textcolor{black}{0}  \\[6pt]
		
		\begin{smallmatrix} \textcolor{black}{i-2\Big[} \end{smallmatrix}& \hspace{-0.45cm}\textcolor{black}{\smash{\underbrace{\; \textcolor{black}{2} \;}_{j-3}}} & \textcolor{black}{2} &   \textcolor{black}{2}&  \textcolor{black}{1}&  \textcolor{black}{0} & \textcolor{black}{0} &  \textcolor{black}{\smash{\underbrace{\; \textcolor{black}{0} \;}_{j-3}}} \\[6pt]
				
\end{pmatrix}
$}\\[3pt]

	%---SUBCASE B12-----
	\item[{\bf B12.}]
	{$\begin{pmatrix}
	0 & 1 & \textcolor{black}{1} & 1 & 1 & 2 \\[3pt]
	0 & 1 & \textcolor{black}{1} & 1 &1 &  2 \\[3pt]
	0 & 1  & \textcolor{black}{1} & 1 & 2 & 1 \\[3pt]
	\tfrac{5}{2} & {-\tfrac{1}{2}} & \textcolor{black}{1} & 2 & 1 & 0 \\[3pt]
	\tfrac{5}{2} & \tfrac{5}{2} & \textcolor{black}{\smash{\underbrace{\; \textcolor{black}{1} \;}_{q-5}}} & 0 & 0 & 0 \\[3pt]
\end{pmatrix}$}\\[3pt]

	%---SUBCASE B13-----
	\item[{\bf B13.}]
	{$ \begin{pmatrix}
	& \hspace{-0.45cm}  {  0} & {  1} & \; \textcolor{black}{1} \;& {  1} & {  1} & {  1} & {  2} \\[3pt]
	& \hspace{-0.45cm}  {  0} & {  1} & \; \textcolor{black}{1} \;& {  1} & {  1} & {  1} & {  2} \\[3pt]
	& \hspace{-0.45cm}  {  0} & {  1} & \; \textcolor{black}{1} \;& {  1} & {  1} & {  2} & {  1} \\[3pt]
	& \hspace{-0.45cm}  {  0} & {  1} & \; \textcolor{black}{1} \;& {  1} & {  1} & {  2} & {  1} \\[3pt]

   	 \begin{smallmatrix}\textcolor{black}{ p - 6 \bm{\Big[} } \end{smallmatrix} &  \hspace{-0.45cm} \textcolor{black}{1} & \textcolor{black}{1} & \; \textcolor{black}{1} \;& \textcolor{black}{1} & \textcolor{black}{1} & \textcolor{black}{1} & \textcolor{black}{1}\\[3pt]

   	 & \hspace{-0.45cm}  {  3} & {-1} & \; \textcolor{black}{1} \;& {  2} & {  2} & {  0} & {  0} \\[3pt]

   	     & \hspace{-0.45cm} {  3}  & {  3} & \textcolor{black}{\smash{\underbrace{\; \textcolor{black}{1} \;}_{q-6}} } & {  0} & {  0} & {  0} & {  0} \\[3pt]
\end{pmatrix}$}\\[3pt]

	%---SUBCASE B14-----
	\item[{\bf B14.}] 
	{$ \begin{pmatrix}
		& \hspace{-0.5cm}\textcolor{black}{0} & {  0} & {  1} & \textcolor{black}{1} & {  1} &{  2} &\textcolor{black}{2}\\[6pt]

		 & \hspace{-0.5cm}\textcolor{black}{0} & {  1}&{  1} & \textcolor{black}{1} & {  1} & {  1}& \textcolor{black}{2} \\[6pt]

		& \hspace{-0.5cm}\textcolor{black}{1} & {  0}&{  1 + y} & \textcolor{black}{1+y} & {  1+y} & {  1}& \textcolor{black}{1} \\[6pt]

		\begin{smallmatrix} \textcolor{black}{p - 5\Big[} \end{smallmatrix} & \hspace{-0.5cm} \textcolor{black}{1} & \textcolor{black}{1} & \textcolor{black}{1} &  \textcolor{black}{1} & \textcolor{black}{1} & \textcolor{black}{1}&\textcolor{black}{1}  \\[3pt]

		 & \hspace{-0.5cm}\textcolor{black}{2} & {  2}&{{-y}} & \textcolor{black}{x} & {  1+x} & {  1}& \textcolor{black}{0} \\[6pt]

		& \hspace{-0.5cm}\textcolor{black}{\smash{\underbrace{\; \textcolor{black}{2} \;}_{j-2}}} & {  2} &  {  2} & \textcolor{black}{\smash{\underbrace{\; \textcolor{black}{1} \;}_{q-2j}} } &   {  0} &  {\bf 0} & \textcolor{black}{\smash{\underbrace{\; \textcolor{black}{0} \;}_{j-2}} } \\[6pt]
		
\end{pmatrix}$} where
\[
y= \frac{1}{q-2j + 2}
\qquad
\mbox{and}
\qquad 
x= \frac{q-2j+1}{q-2j + 2}.
\]
	%---SUBCASE B15-----
	\item[{\bf B15.}]
	{$\begin{pmatrix}
		\begin{smallmatrix} \textcolor{black}{i-2\Big[} \end{smallmatrix}& \hspace{-0.5cm}\textcolor{black}{0} & \textcolor{black}{0} & \textcolor{black}{1} &  \textcolor{black}{2} & \textcolor{black}{2}  \\[6pt]

		 & \hspace{-0.5cm}\textcolor{black}{0} & { 1} & \textcolor{black}{1} & { 1} &  \textcolor{black}{2}  \\[6pt]

		 & \hspace{-0.5cm}\textcolor{black}{0} & { 1+x} & \textcolor{black}{1} & { 1+y} &  \textcolor{black}{2}  \\[6pt]

		\begin{smallmatrix} \textcolor{black}{p - 2i\Big[} \end{smallmatrix} & \hspace{-0.5cm} \textcolor{black}{1} & \textcolor{black}{x} & \textcolor{black}{1} &  \textcolor{black}{1+y} & \textcolor{black}{1}   \\[3pt]
		
		  & \hspace{-0.5cm}\textcolor{black}{2} & {-y} & \textcolor{black}{1} & { 1+y} &  \textcolor{black}{1}  \\[6pt]

		\begin{smallmatrix} \textcolor{black}{i-1\Big[} \end{smallmatrix}& \hspace{-0.45cm}\textcolor{black}{\smash{\underbrace{\; \textcolor{black}{2} \;}_{j-1}} } & { 2} & \textcolor{black}{\smash{\underbrace{\; \textcolor{black}{1} \;}_{q-2j}} }& { 0} &  \textcolor{black}{\smash{\underbrace{\; \textcolor{black}{0} \;}_{j-1}}} \\[6pt]
			
\end{pmatrix}$} where 
\[
y= \frac{1}{p-2i + 2}
\qquad
\mbox{and} 
\qquad
x= \frac{p-2i+1}{p-2i + 2}.
\]
\end{itemize}

The subcases for case (b3a) are the following:
%---LIST OF SUBCASES FOR CASE (b3a)------
\begin{itemize}
	%---SUBCASE C1-----
	\item[{\bf C1.}] ${\sum\limits_{h=1}^j b_{2h} < \sum\limits_{h=1}^j b_{1h} } \quad 1 \leq j \leq \floor*{\tfrac{q-1}{2}},$
	%---SUBCASE C2-----
	\item[{\bf C2.}] ${\sum\limits_{h=1}^{ \hat{j} } b_{2h} < \sum\limits_{h=1}^{\hat{j}} b_{1h}} \quad \hat{j}= \floor*{\tfrac{q+1}{2}}, \text{ $q$ odd};$
	%---SUBCASE C3-----
	\item[{\bf C3.}] ${\sum\limits_{h=1}^{ \hat{j} } b_{2h} < \sum\limits_{h=1}^{\hat{j}} b_{1h}} \text{ with } \hat{j}= \floor*{\tfrac{q}{2}},\text{ $q$ even}; $
	%---SUBCASE C4-----
	\item[{\bf C4.}] ${\sum\limits_{h=1}^i b_{i+1,h} < \sum\limits_{h=1}^j b_{ih}} \; \text{with} \; 1 \leq j \leq \floor*{\tfrac{q-1}{2}}, \, 2 \leq i \leq p-2 $
	%---SUBCASE C5-----
	\item[{\bf C5.}] ${\sum\limits_{h=1}^{ \hat{j} } b_{i+1,h} < \sum\limits_{h=1}^{\hat{j}} b_{ih}} \; \hat{j}= \floor*{\tfrac{q+1}{2}}, \,2 \leq i \leq p-2; \text{ $q$ odd};$
	%---SUBCASE C6-----
	\item[{\bf C6.}] ${\sum\limits_{h=1}^{ \hat{j} } b_{i+1,h} < \sum\limits_{h=1}^{\hat{j}} b_{ih} } \; \hat{j}= \floor*{\tfrac{q}{2}}, \,2 \leq i \leq p-2; \text{ $q$ even};$
	%---SUBCASE C7-----
	\item[{\bf C7.}] ${\sum\limits_{h=1}^i b_{ph} < \sum\limits_{h=1}^j b_{p-1,h}  } \text{ with } 1 \leq j \leq \floor*{\tfrac{q-1}{2}}, \, {q,p\geq4};$
	%---SUBCASE C8-----
	\item[{\bf C8.}] ${\sum\limits_{h=1}^{ \hat{j} } b_{p h} < \sum\limits_{h=1}^{\hat{j}} b_{p-1,h} } \text{ with } \hat{j}= \floor*{\tfrac{q+1}{2}}, \, {q,p\geq4} \text{ $q$ odd};$
	%---SUBCASE C9-----
	\item[{\bf C9.}] ${\sum\limits_{h=1}^{ \hat{j} } b_{ph} < \sum\limits_{h=1}^{\hat{j}} b_{p-1,h}} \text{ with }\hat{j}= \floor*{\tfrac{q}{2}}, \, {q,p\geq4} \text{ $q$ even}.$
\end{itemize}
The following is the associated list of matrices for the subcases of case (b3a) listed above.
%---LIST OF MATRICES FOR CASE (b2)----
\begin{itemize}
	%---SUBCASE C1-----
	\item[{\bf C1.}]
	{$\begin{pmatrix}
		& \hspace{-0.5cm}\textcolor{black}{0} & { \tfrac{1}{2}} & { \tfrac{1}{2}} & \textcolor{black}{1} & \textcolor{black}{2} \\[3pt]

		& \hspace{-0.5cm}{0} & { 0} & { 1} & \textcolor{black}{1} & \textcolor{black}{2} \\[3pt]

		\begin{smallmatrix} \textcolor{black}{p-4 \Big[ }\end{smallmatrix} & \hspace{-0.5cm} \textcolor{black}{1} & \textcolor{black}{1} & \textcolor{black}{1} &  \textcolor{black}{1} & \textcolor{black}{1}  \\[3pt]

		& \hspace{-0.5cm}\textcolor{black}{2} & { \tfrac{3}{2}} & { \tfrac{3}{2}} & \textcolor{black}{1}& \textcolor{black}{0} \\[3pt]
		
		& \hspace{-0.5cm}\textcolor{black}{\smash{\underbrace{\; \textcolor{black}{2} \;}_{j-1}}} & { 2} & { 1} & \textcolor{black}{\smash{\underbrace{\; \textcolor{black}{1} \;}_{q-2j-1}}} & \textcolor{black}{\smash{\underbrace{\; \textcolor{black}{0} \;}_{j}}} \\[3pt]
\end{pmatrix}$}\\[3pt]

	%---SUBCASE C2-----
	\item[{\bf C2.}]
	{$\begin{pmatrix}
		& \hspace{-0.5cm}\textcolor{black}{0} & { 0} & { 1} & { 1} & { 1} & \textcolor{black}{2} \\[3pt]

		& \hspace{-0.5cm}{0} & { 1} & { 0} &  { 0} &{ 2} & \textcolor{black}{2} \\[3pt]

		\begin{smallmatrix} \textcolor{black}{p-4 \Big[ }\end{smallmatrix} & \hspace{-0.5cm} \textcolor{black}{1} & \textcolor{black}{1} & \textcolor{black}{1} &  \textcolor{black}{1} & \textcolor{black}{1} & \textcolor{black}{1} \\[3pt]

		& \hspace{-0.5cm}\textcolor{black}{2} & { 1} & { 1} & { 2} & { 1} & \textcolor{black}{0} \\[3pt]
		
		& \hspace{-0.5cm}\textcolor{black}{\smash{\underbrace{\; \textcolor{black}{2} \;}_{\hat{j}-3}}} & { 2} & { 2} &  { 1} & { 0} & \textcolor{black}{\smash{\underbrace{\; \textcolor{black}{0} \;}_{\hat{j}-2}}} \\[3pt]
\end{pmatrix}$}\\[3pt]

	%---SUBCASE C3-----
	\item[{\bf C3.}]
	{$\begin{pmatrix}
		& \hspace{-0.5cm}\textcolor{black}{0} & { 1} & { 1} &  \textcolor{black}{2} \\[3pt]

		& \hspace{-0.5cm}{0} & { 0} & { 2} & \textcolor{black}{2} \\[3pt]

		\begin{smallmatrix} \textcolor{black}{p-4 \Big[ }\end{smallmatrix} & \hspace{-0.5cm} \textcolor{black}{1} & \textcolor{black}{1} & \textcolor{black}{1} &  \textcolor{black}{1} \\[3pt]

		& \hspace{-0.5cm}\textcolor{black}{2} & { 1} & { 1} & \textcolor{black}{0} \\[3pt]
		
		& \hspace{-0.5cm}\textcolor{black}{\smash{\underbrace{\; \textcolor{black}{2} \;}_{\hat{j}-1}}} & { 2} & { 0}  & \textcolor{black}{\smash{\underbrace{\; \textcolor{black}{0} \;}_{\hat{j}-1}}} \\[3pt]
\end{pmatrix}$}\\[3pt]

	%---SUBCASE C4-----
	\item[{\bf C4.}]
	{$\begin{pmatrix}
		& \hspace{-0.5cm}\textcolor{black}{0} & { \tfrac{1}{2}} & { \tfrac{1}{2}} & \textcolor{black}{1} & \textcolor{black}{2} \\[3pt]

		\begin{smallmatrix} \textcolor{black}{i-1 \Big[ }\end{smallmatrix} & \hspace{-0.5cm} \textcolor{black}{1} & \textcolor{black}{1} & \textcolor{black}{1} &  \textcolor{black}{1} & \textcolor{black}{1}  \\[3pt]
		
		& \hspace{-0.5cm}{1} & { \tfrac{1}{2}} & \textcolor{black}{ \tfrac{3}{2}} & \textcolor{black}{1} & \textcolor{black}{1} \\[3pt]

		\begin{smallmatrix} \textcolor{black}{p-i-2 \Big[ }\end{smallmatrix} & \hspace{-0.5cm} \textcolor{black}{1} & \textcolor{black}{1} & \textcolor{black}{1} &  \textcolor{black}{1} & \textcolor{black}{1}  \\[3pt]

		& \hspace{-0.5cm}\textcolor{black}{\smash{\underbrace{\; \textcolor{black}{2} \;}_{j-1}}} & { 2} & { 1} & \textcolor{black}{\smash{\underbrace{\; \textcolor{black}{1} \;}_{q-2j-1}}} & \textcolor{black}{\smash{\underbrace{\; \textcolor{black}{0} \;}_{j}}} \\[3pt]
\end{pmatrix}$} \\[3pt]

	%---SUBCASE C5-----
	\item[{\bf C5.}]
	{$\begin{pmatrix}
		& \hspace{-0.5cm}\textcolor{black}{0} & { 0} & { {2}} & \textcolor{black}{2} & { 3} \\[3pt]

		\begin{smallmatrix} \textcolor{black}{i-2 \Big[ }\end{smallmatrix} & \hspace{-0.5cm} \textcolor{black}{1} & \textcolor{black}{1} & \textcolor{black}{1} &  \textcolor{black}{1} & \textcolor{black}{1}  \\[3pt]
				
		& \hspace{-0.5cm}\textcolor{black}{1} & { 2} & { {0}} & \textcolor{black}{1} & { 1} \\[3pt]

		\begin{smallmatrix} \textcolor{black}{p-i-2 \Big[ }\end{smallmatrix} & \hspace{-0.5cm} {1} & { 1} & \textcolor{black}{1} &  \textcolor{black}{1} & \textcolor{black}{1}  \\[3pt]

		& \hspace{-0.5cm}\textcolor{black}{1} & { 1} & { {2}} & \textcolor{black}{1} & { 0} \\[3pt]

		& \hspace{-0.5cm}\textcolor{black}{\smash{\underbrace{\; \textcolor{black}{2} \;}_{\hat{j}-1}}} & { 1} & { 0} & \textcolor{black}{\smash{\underbrace{\; \textcolor{black}{0} \;}_{\hat{j}-3}}} &  { 0} \\[3pt]
\end{pmatrix}$}\\[3pt]

	%---SUBCASE C6-----
	\item[{\bf C6.}]
	{$\begin{pmatrix}
		& \hspace{-0.5cm}\textcolor{black}{0} & { 0} & { {2}}  &  \textcolor{black}{2}  \\[3pt]

		\begin{smallmatrix} \textcolor{black}{i-2 \Big[ }\end{smallmatrix} & \hspace{-0.5cm} \textcolor{black}{1} & \textcolor{black}{1} &  \textcolor{black}{1} & \textcolor{black}{1}  \\[3pt]
				
		& \hspace{-0.5cm}\textcolor{black}{1} & { 2} & { {0}} & \textcolor{black}{1}  \\[3pt]

		\begin{smallmatrix} \textcolor{black}{p-i-1 \Big[ }\end{smallmatrix} & \hspace{-0.5cm} {1} & { 1} &   \textcolor{black}{1} & \textcolor{black}{1}  \\[3pt]

		& \hspace{-0.5cm}\textcolor{black}{\smash{\underbrace{\; \textcolor{black}{2} \;}_{\hat{j}-1}}} & { 1} & { 1} & \textcolor{black}{\smash{\underbrace{\; \textcolor{black}{0} \;}_{\hat{j}-1}}}  \\[3pt]
\end{pmatrix}$}\\[3pt]

	%---SUBCASE C7-----
	\item[{\bf C7.}]
	{$\begin{pmatrix}
		& \hspace{-0.5cm}\textcolor{black}{0} & { 0} & { 1} & \textcolor{black}{1} & \textcolor{black}{2} \\[3pt]

		& \hspace{-0.5cm}\textcolor{black}{0} & { 0} & { 1} & \textcolor{black}{1} & \textcolor{black}{2} \\[3pt]

		\begin{smallmatrix} \textcolor{black}{p-4 \Big[ }\end{smallmatrix} & \hspace{-0.5cm} \textcolor{black}{1} & \textcolor{black}{1} & \textcolor{black}{1} &  \textcolor{black}{1} & \textcolor{black}{1}  \\[3pt]

		& \hspace{-0.5cm}\textcolor{black}{2} & { \tfrac{5}{2}} & { \tfrac{1}{2}} & \textcolor{black}{1} & \textcolor{black}{0} \\[3pt]
		
		& \hspace{-0.5cm}\textcolor{black}{\smash{\underbrace{\; {2} \;}_{j-1}}} & { {\tfrac{3}{2}}} & { \tfrac{3}{2}} & \textcolor{black}{\smash{\underbrace{\; \textcolor{black}{1} \;}_{q-2j-1}}} & \textcolor{black}{\smash{\underbrace{\; \textcolor{black}{0} \;}_{j}}} \\[3pt]
\end{pmatrix}$} \\[3pt]

	%---SUBCASE C8-----
	\item[{\bf C8.}]
	{$\begin{pmatrix}
		& \hspace{-0.5cm}\textcolor{black}{0} & { 1} & { 2} & \textcolor{black}{2} \\[3pt]

		& \hspace{-0.5cm}\textcolor{black}{0} & { \tfrac{3}{2}} & { \tfrac{3}{2}} & \textcolor{black}{2} \\[3pt]

		\begin{smallmatrix} \textcolor{black}{p-4 \Big[ }\end{smallmatrix} & \hspace{-0.5cm} \textcolor{black}{1} & \textcolor{black}{1} & \textcolor{black}{1} &  \textcolor{black}{1}   \\[3pt]

		& \hspace{-0.5cm}\textcolor{black}{2} & { 1} & { 0} & \textcolor{black}{0} \\[3pt]
		
		& \hspace{-0.5cm}\textcolor{black}{\smash{\underbrace{\; {2} \;}_{\hat{j}-1}}} & { {\tfrac{1}{2}}} & { \tfrac{1}{2}}  & \textcolor{black}{\smash{\underbrace{\; \textcolor{black}{0} \;}_{\hat{j}-2}}} \\[3pt]
\end{pmatrix}$} \\[3pt]

	%---SUBCASE C9-----
	\item[{\bf C9.}]
	{$\begin{pmatrix}
		& \hspace{-0.5cm}\textcolor{black}{0} & { 0} & { 2} & \textcolor{black}{2} \\[3pt]

		& \hspace{-0.5cm}\textcolor{black}{0} & { 1} & { 1} & \textcolor{black}{2} \\[3pt]

		\begin{smallmatrix} \textcolor{black}{p-4 \Big[ }\end{smallmatrix} & \hspace{-0.5cm} \textcolor{black}{1} & \textcolor{black}{1} & \textcolor{black}{1} &  \textcolor{black}{1}   \\[3pt]

		& \hspace{-0.5cm}\textcolor{black}{2} & { 2} & { 0} & \textcolor{black}{0} \\[3pt]
		
		& \hspace{-0.5cm}\textcolor{black}{\smash{\underbrace{\; {2} \;}_{\hat{j}-1}}} & { {1}} & { 1}  & \textcolor{black}{\smash{\underbrace{\; \textcolor{black}{0} \;}_{\hat{j}-1}}} \\[3pt]
\end{pmatrix}$} \\[3pt]
\end{itemize}
Together with Lemma~\ref{Lem:1}, these subcases and their corresponding matrices complete the proof of the theorem. \hfill $\square$ 

%---SUBSECTION: Proof of Theorem~\ref{Th:ext}
\subsection{Proof of Theorem \ref{Th:ext}.}
Every $C \in \UDC_{p,q}$ is a discrete copula on $I_p \times I_q$. 
Thus, according to \cite[Lemma 2.3.5]{nelsen_06}, the checkerboard extension $\tilde{C}$ of $C$ which is defined as
\[
\tilde{C}(u,v) = (1 - \lambda_u) (1 - \mu_v) c_{ij} + (1 - \lambda_u) \mu_v c_{i,j+1}  + \lambda_u (1 - \mu_v) c_{i+1, j} + \lambda_u \mu_v c_{i+1, j+1}
\]
where $\tfrac{i}{p} \leq u \leq \tfrac{i + 1}{p}$, $\tfrac{j}{q} \leq v \leq \tfrac{j + 1}{q}$, and
\[ 
\lambda_u = \left\{ \begin{array}{lc} (u - \tfrac{i}{p})p & u > \tfrac{i}{p} \\
1 & u = \tfrac{i}{p} 
\end{array} \right. 
\quad
\text{  and  }
\quad
 \mu_v = \left\{ \begin{array}{lc} (v - \tfrac{j}{q})q & v > \tfrac{j}{q} \\
1 & v = \tfrac{j}{q} 
\end{array}\right.
\]
is a copula on $[0,1]^2$, whose restriction on $I_p \times I_q$ is $C$.
We now show that for any $C \in \UDC_{p,q}$,  $\tilde{C}$ is an ultramodular copula; i.e., $\tilde{C}$ has convex horizontal and vertical (coordinatewise) sections.
We here focus on any arbitrary horizontal section $C_{a}: u \mapsto \tilde{C}(u,a)$ with $a \in [0,1]$ and prove that it is a convex function. 
The same argument can be used to prove the convexity of an arbitrary vertical section.
$C_{a}$ is a $p$-piecewise continuous function. 
Therefore, to prove its convexity it is sufficient to show the \emph{Jensen convexity}, i.e., for $u_1, u_2 \in[0,1]$ 
\begin{equation}
\label{Eq:conv} 
C_{a}\Big(\tfrac{u_1}{2} + \tfrac{u_2}{2} \Big) \leq  \tfrac{1}{2} C_{a}(u_1) + \tfrac{1}{2} C_{a}(u_2) .
\end{equation}
Without loss of generality, we assume $ \tfrac{j}{q} < a < \tfrac{j+1}{q}$ and define $\mu_a = (a - \tfrac{j}{q})\; q$. 
We then proceed by induction on the number $M$ of intervals that contain $[u_1,u_2]$.
To do this, we consider a few cases.

{\bf \textsc{CASE} $\text{M=1}$.}  
Let us consider $ \tfrac{i}{p} < u_1, u_2 < \tfrac{i+1}{p}$, and $ u_3 = \tfrac{u_1}{2} + \tfrac{u_2}{2}$, for $i\in \langle p-1\rangle$. 
By definition, $ \lambda_1 = pu_1 - i$ and $ \lambda_2 = pu_2 - i$. 
Hence, 
\[
\lambda_3 = \tfrac{p u_1}{2} + \tfrac{p u_2}{2} - i = \tfrac{\lambda_1}{2} + \tfrac{\lambda_2}{2}
\]
By construction, we can then express Eq.(\ref{Eq:conv}) as follows.
\[
\begin{array}{lll}
C_a(u_3) &= &  (1 - \lambda_3 )(1-\mu_a) c_{ij} +  (1 - \lambda_3 )\mu_a c_{i, j+1} 		\\[6pt]
		 && \hspace{15pt} + \lambda_3 (1-\mu_a)c_{i+1, j} + \lambda_3 \mu_a c_{i+1, j+1}		\\[6pt]
         &=&  \frac{C_{a}(u_1)}{2}  + \frac{C_{a}(u_2)}{2}\\
 \end{array}
\]

{\bf \textsc{CASE} $\text{M=2}$.}
Let us consider $ \tfrac{i}{p} < u_1 < \tfrac{i+1}{p} < u_2 < \tfrac{i+2}{p}$ for $i\in \langle p-1\rangle$. 
Then, ${\tfrac{C_{a}(u_1)}{2} } + \tfrac{C_{a}(u_2)}{2}$ can be written as
\[
\begin{array}{ll}
\tfrac{C_{a}(u_1)}{2} + \tfrac{C_{a}(u_2)}{2} = & (1-\mu_a) \left(\tfrac{1}{2} - \tfrac{\lambda_1}{2}\right)c_{ij} +\left(\tfrac{1}{2} - \tfrac{\lambda_1}{2}\right)\mu_a c_{i, j+1} 	\\
  	&\hspace{15pt} + \tfrac{\lambda_1}{2}(1-\mu_a)c_{i+1, j} + \mu_a \tfrac{\lambda_1}{2} c_{i+1, j+1}  \\[6pt]
 &\hspace{15pt} + \tfrac{1}{2}(1-\mu_a) c_{i+1, j} + \tfrac{1}{2} (1-\mu_a)c_{i+1, j+1} \\ [6pt]
 & \hspace{15pt} +(1 - \mu_a)\tfrac{\lambda_2}{2} \left(c_{i+2, j} - c_{i+1, j}\right)	\\ [6pt]
 &\hspace{15pt} + \mu_a \tfrac{\lambda_2}{2} \left( c_{i+2, j+1}  - c_{i+1, j+1}  \right)\\
\\
 \end{array}
\]
If $\tfrac{i}{p} < u_3 < \tfrac{i+1}{p}$, then $\lambda_3 = \tfrac{p u_1}{2} + \tfrac{p u_2}{2} - i = \tfrac{\lambda_1}{2} + \tfrac{\lambda_2}{2} + \tfrac{1}{2}$. 
Thus, from inequalities (d3a) and (d3b), we have that
\[
{ c_{i+2, j+1} -  c_{i+1, j+1} \geq \; \; c_{i+1, j+1} -  c_{i, j+1} \geq \; \; c_{i+1, j} -  c_{ij}. }
\]
Thus, it follows that
\[
\begin{array}{lll}

\tfrac{C_{a}(u_1)}{2} + \tfrac{C_{a}(u_2)}{2}  & \geq & 
 \left[(1-\mu_a)\left(\tfrac{1}{2} - \tfrac{\lambda_1}{2}\right)  - (1-\mu_a) \tfrac{\lambda_2}{2} \right]c_{ij}	\\[6pt]
 && \hspace{15pt}+ \left[\mu_a\left(\tfrac{1}{2} - \tfrac{\lambda_1}{2}\right) - \mu_a \tfrac{\lambda_2}{2} \right] c_{i, j+1}  \\[6pt]
 &&\hspace{15pt} + \left[(1-\mu_a)\tfrac{\lambda_1}{2} + \tfrac{1}{2} (1-\mu_a) + (1-\mu_a) \tfrac{\lambda_2}{2}  \right] \cdot c_{i+1, j}	\\[6pt]
 && \hspace{15pt} + \left[\mu_a \tfrac{\lambda_1}{2} + \tfrac{1}{2} \mu_a + \mu_a \tfrac{\lambda_2}{2} \right] c_{i+1, j+1} \\[6pt]
 &=& C_{a}\left( \tfrac{u_1}{2}  + \tfrac{u_2}{2} \right)\\[6pt]
 \end{array}
\]
Assuming $ \tfrac{i+1}{p} < u_3 < \tfrac{i+2}{p}$. 
One has $\lambda_3 = \tfrac{p u_1}{2} + \tfrac{p u_2}{2} - i = \tfrac{\lambda_1}{2} + \tfrac{\lambda_2}{2} - \tfrac{1}{2}.$ 
Conditions (d3a) and (d3b) imply the following inequalities for $k \in \{j, j+1 \}$
\[ 
c_{i, k} \geq 2 \;c_{i+1, k} - c_{i+2, k}
\]
The result can therefore be derived as follows.
\[
\begin{array}{lll}
\tfrac{C_{a}(u_1)}{2} + \tfrac{C_{a}(u_2)}{2}  & \geq & 
 \left[ (1 - \mu_a) \left( 1 - \lambda_1 + \tfrac{\lambda_1}{2}+ \tfrac{1}{2} - \tfrac{\lambda_2}{2}\right) \right] c_{i+2, j+1} \\[6pt]
& & \hspace{15pt} + \left[ (1 - \mu_a)  \left( - \tfrac{1}{2} + \tfrac{\lambda_1}{2} + \tfrac{\lambda_2}{2}\right) \right] c_{i+2, j} 	\\[6pt]
& & \hspace{15pt} + \left[ \mu_a \left(1 - \lambda_1 + \tfrac{\lambda_1}{2} + \tfrac{1}{2} - \tfrac{\lambda_2}{2}\right) \right] c_{i+1, j+1}  \\[6pt]
& & \hspace{15pt} + \left[ \mu_a \left( - \tfrac{1}{2} + \tfrac{\lambda_1}{2} + \tfrac{\lambda_2}{2}\right) \right] c_{i+2, j+1}  \\[6pt]
& = & C_{a}\left( \tfrac{u_1}{2}  + \tfrac{u_2}{2} \right) \\[6pt]
 \end{array}
\]

{\bf \textsc{CASE} $M=N \leq p$.}
Let us assume the result is true for $N-1$ intervals. 
In order to prove that $C_a$ is convex, we only need to show that the last two intervals of the partition attach in a convex way.
Therefore, we can restrict ourselves to the situation where $ \tfrac{N-2}{p} < u_1 <\tfrac{N-1}{p} < u_2 < \tfrac{N}{p}$. The thesis follows from case M=2.
\hfill $\qed$

%---SECTION: Appendix II: Proofs for Section~\ref{sec: the polytopes of discrete quasi-copulas and convex discrete quasi-copulas}----
\section*{Appendix B: Proofs for Section~\ref{sec: the polytopes of discrete quasi-copulas and convex discrete quasi-copulas}}
\label{sec: appendix B}

%---LEMMA: Containment for generalized alternating sign matrix polytope----
\begin{lemma}
\label{lem: containment for generalized alternating sign matrix polytope}
Suppose that $3 \leq p< q$ with $q = kp+r$ for $0\leq r < p$ and that $[c_{ij}]\in\R^{(p+1)\times(q+1)}$ satisfies all of (a1), (a2), and (a3).  Then $[c_{ij}]\in\ASM_{p,q}$.  
\end{lemma}

%%---SUBSECTION: Proof of Lemma~\ref{lem: containment for generalized alternating sign matrix polytope}----
%\subsection{Proof of Lemma~\ref{lem: containment for generalized alternating sign matrix polytope}.}
\begin{proof}
Recall that for $p<q$ with $q = pk+r$ with $0\leq r < p$ the alternating sign matrix polytope $\ASM_{p,q}$ is defined by the collection of inequalities 
\begin{enumerate}
	\item $\sum_{\ell = 1}^p x_{\ell j} = p$; $\sum_{h = 1}^q x_{ih} = q$ for $i\in[p]$ and $j\in[q]$,
	\item $0\leq\sum_{\ell = 1}^i x_{\ell j}\leq p$ for all $i\in[p]$ and $j\in[q]$,
	\item $0\leq \sum_{h = 1}^j x_{ih} \leq q$ for all $i\in[p]$ and $j\in[q]$.
\end{enumerate}
Using the equalities (1), we can transform the inequalities (2) and (3) into the two families
\begin{enumerate}
	\item[2(a).]  $0\leq \sum_{\ell = 1}^i x_{\ell j}$ for all $i\in[p]$ and $j\in[q]$,
	\item[2(b).]  $0\leq \sum_{\ell = i+1}^p x_{\ell j}$ for all $i\in[p]$ and $j\in[q]$,
	\item[3(a).]  $0\leq \sum_{h = 1}^j x_{ih}$ for all $i\in[p]$ and $j\in[q]$, and
	\item[3(b).]  $0\leq \sum_{h = j+1}^q x_{ih}$ for all $i\in[p]$ and $j\in[q]$.
\end{enumerate}
By symmetry, it suffices to determine which inequalities among 2(a) and 3(a) are necessary and then take their symmetric opposites from among 2(b) and 3(b) as well.  

Notice first that since the full column sums are always equal to $q >0$, then the equality $\sum_{\ell = 1}^p x_{\ell j}=0$ yields the empty set.  
Thus, the case when $i = p$ for $j\in[q]$ is not facet-defining.  
Similarly, this is true for the case when $j = q$ and $i\in[p]$.  
Next notice that the inequalities $x_{\ell1}\geq0$ for all $\ell\in[p]$ imply that $\sum_{\ell = 1}^i x_{\ell j}\geq 0$ for $i\in\{2,\ldots, p-1\}$.  
Thus, the inequalities of type 2(a) are not facet-defining when $i\in\{2,\ldots, p-1\}$ and $j = 1$.  
Similarly, the inequalities of type 3(a) are not facet-defining when $i =1$ and $j\in\{2,\ldots, q-1\}$.  
Thus, we now know that the minimal $H$ representation of $\ASM_{p,q}$ is contained within the collection of inequalities
\begin{enumerate}
	\item[2(a).]  $0\leq \sum_{\ell = 1}^i x_{\ell j}$ for all $i\in\{1,\ldots,p-1\}$ and $j\in\{2,\ldots,q-1\}$,
	\item[2(b).]  $0\leq \sum_{\ell = i+1}^p x_{\ell j}$ for all $i\in\{1,\ldots,p-1\}$ and $j\in\{2,\ldots,q-1\}$,
	\item[3(a).]  $0\leq \sum_{h = 1}^j x_{ih}$ for all $i\in\{2,\ldots,p-1\}$ and $j\in\{1,\ldots,q-1\}$, and
	\item[3(b).]  $0\leq \sum_{h = j+1}^q x_{ih}$ for all $i\in\{2,\ldots,p-1\}$ and $j\in\{1,\ldots,q-1\}$.
\end{enumerate}
To complete the proof, it remains to show that the inequalities of type 3(a) $\sum_{h=1}^jx_{ih}\geq0$ are redundant (i.e. not facet-defining) whenever $i\in\{2,\ldots,p-1\}$ and $j\in\{q-k,\ldots,q-1\}$.
Notice first that when $p\leq q$ and $[c_{ij}]\in\ASM_{p,q}$ then $c_{ij}\leq p$ for all $i\in[p]$ and $j\in[q]$.  
To see this fact, recall that $\ASM_{p,q}$ is defined by the inequalities listed in (1), (2), and (3) above.  
So, if there existed some $c_{ij}>p$, then since $0\leq \sum_{\ell = 1}^{i=1}c_{\ell j}$, it would follow that $\sum_{\ell = 1}^i c_{\ell j} > p$, which contradicts the above inequalities defining $\ASM_{p,q}$.  

Now, let $i\in\{2,\ldots, p-1\}$.  
Since $x_{ih}\leq q = pk+r$ for all $h\in[q]$ then $\sum_{h=j+1}^q x_{ih}\leq q$ for all $j\in\{q-k,\ldots, q-1\}$.  
Thus, since $\sum_{h=1}^q x_{ih} = q$, it follows that $\sum_{h=1}^j x_{ih}\geq 0$, as desired.  

Notice that for the symmetry argument to work, we must not apply it to the corner inequalities; i.e., $x_{11}\geq 0$, $x_{1p} \geq 0$, $x_{1q}\geq 0$ and $x_{pq}\geq0$.  
Thus, these inequalities are counted separately from the rest within (a1).  
This completes the proof.
\end{proof}

Given Lemma~\ref{lem: containment for generalized alternating sign matrix polytope}, to prove Theorem~\ref{thm: facets of the generalized alternating sign matrix polytope} it remains to show that for each inequality in the list (a1), (a2), and (a3), there exists a point $[c_{ij}]\in\R^{p\times q}$ satisfying all inequalities in the list with the exception of the chosen one.

%---SUBSECTION: Proof of Theorem~\ref{thm: facets of the generalized alternating sign matrix polytope}----
\subsection*{Proof of Theorem~\ref{thm: facets of the generalized alternating sign matrix polytope}.}
By Lemma~\ref{lem: containment for generalized alternating sign matrix polytope}, we know that the minimal $H$-representation of $\ASM_{p,q}$ for $p\neq q$ is contained within the collection of inequalities (a1), (a2) and (a3).  
We here prove that inequalities (a1), (a2) and (a3) are exactly the minimal $H$-representation of $\ASM_{p,q}$ for $p\neq q$.
To do this, it suffices to show that for each inequality in the list there exists a matrix $[c_{ij}]\in\R^{p\times q}$ that does not satisfy the chosen inequality but satisfies all other inequalities among (a1), (a2), and (a3).  
The matrices are given as follows.
The matrix $P$ 

\[
P= \begin{pmatrix}
      & \hspace{-0.5cm} -1 & 1 & 1 & 3														\\
   \begin{smallmatrix} {p-3 \Big[ }\end{smallmatrix} & \hspace{-0.5cm} 1 & 1 & 1 & 1		\\
      & \hspace{-0.5cm}1		&	1	&	2		& 0	 									\\
      & \hspace{-0.5cm}3		&	{\smash{\underbrace{\; 1 \;}_{q-3}}} & 0 & 0            \\[6pt]
\end{pmatrix}
\in\R^{p\times q}
\]  
can be seen to satisfy all inequalities among (a1), (a2), and (a3) except for $x_{11}\geq 0$.
By permuting the columns of this matrix and flipping the matrix horizontally, we see the desired matrices for the other inequalities listed in (a1).  
For the conditions listed in (a1), the analogous matrix for the inequality $\sum_{\ell = 1}^i x_{\ell 2} \geq 0$ is the matrix
\[
\begin{pmatrix}
A					&	{\bf 1}_{i\times (q-i-2)}		\\
{\bf 1}_{(p-i-1)\times (i+2)}	&	{\bf 1}_{(p-i-1) \times (q-i-2)}	\\
\end{pmatrix}
\in\R^{p\times q},
\]
where $A$ is the block matrix
$
\begin{pmatrix} 
B	&	C	\\
\end{pmatrix}
\in\R^{(i+1)\times (i+2)},
$
with $B, C$ as follows
\[
B = \begin{pmatrix}
{\bf 1}_{(i-1)\times 1}		&	{\bf 0}_{(i-1)\times 1}	\\
2		&	-1	\\
0		&	i+2	\\
\end{pmatrix}
\in\R^{(i+1)\times 2},
\,
C = \begin{pmatrix}
	{\bf 1}_{i\times i}+ I_i	\\
    {\bf 0}_{1\times i}	\\
\end{pmatrix}
\in\R^{(i+1)\times i}.
\]

Permuting the columns and flipping this matrix horizontally then recovers the matrices for the other inequalities listed in (a2). 
Similarly, for the inequality $\sum_{h = 1}^j x_{2h}\geq0$ listed in (a3), we use the matrix
\[
\begin{pmatrix}
A					&	{\bf 1}_{3\times (q-2j+2)}		\\
{\bf 1}_{(p-3)\times (2j-2)}	&	{\bf 1}_{(p-3) \times (q-2j+2)}	\\
\end{pmatrix}
\in\R^{p\times q},
\]
where $A$ is the block matrix 
$
\begin{pmatrix} 
B	&	C	&	D	\\
\end{pmatrix}
\in\R^{3\times (2j-2)},
$
where $B, C$, and $D$ are
\[
B = \begin{pmatrix}
2		&	2		&	\cdots		&	2	\\
0		&	0		&	\cdots		&	0	\\
1		&	1		&	\cdots		&	1	\\
\end{pmatrix},
\, 
D = \begin{pmatrix}
0		&	0		&	\cdots		&	0	\\
2		&	2		&	\cdots		&	2	\\
1		&	1		&	\cdots		&	1	\\
\end{pmatrix}
\in\R^{3\times (j-2)},
\text{ and }
C = \begin{pmatrix}
2		&	0	\\
-1		&	3	\\
2		&	0	\\
\end{pmatrix}
\]
Here, permuting the rows and flipping the matrix along its vertical axis produces the remaining desired matrices.
Collectively, these matrices combined with Lemma~\ref{lem: containment for generalized alternating sign matrix polytope} complete the proof.
\hfill $\square$

%---LEMMA: Containment for Polytope of Convex Discrete Quasi-copulas----
\begin{lemma}
\label{Lem:DQua}
Suppose that $[c_{ij}]\in\R^{(p+1)\times(q+1)}$ satisfies all of (v1), (v3a), and (v3b) as well as the equalities
\[
c_{0k}=0, \, c_{pk}=\frac{k}{q}, \,  c_{h0}=0, \, c_{hq}=\frac{h}{p} \, \mbox{for all } h \in \langle p\rangle, k \in \langle q\rangle.
\]  
Then $[c_{ij}]\in\CDQ_{p,q}$.  
\end{lemma}

%%---SUBSECTION: Proof of Lemma~\ref{Lem:DQua}}----
%\subsection{Proof of Lemma~\ref{Lem:DQua}.}
\begin{proof}
Let us consider $C = [c_{ij}]\in\R^{(p+1)\times(q+1)}$ that satisfies all of the inequalities (v1), (v3a), and (v3b) as well as those equalities stated in the lemma. 
Then $C$ satisfies the equalities (q1). 
The proof of Lemma \ref{Lem:1} also shows that $C$ meets the following requirements for $i\in\langle p-1\rangle$ and $j \in \langle q-1\rangle$.
\[
\begin{array}{lll}
      \textbf{a} & (1) \; c_{1, j+1} - c_{1, j} \geq  0, & (2) \quad c_{i+1, 1} - c_{ij} \geq  0\\[6pt]
      \textbf{b.} & (1) \; c_{p-1, j+1} - c_{p-1, j} \leq \tfrac{1}{q}, & (2) \quad  c_{i+1, q-1} -c_{i+1, q-1} \leq \tfrac{1}{p}\\[6pt]
   \end{array} 
\]
Conditions \textbf{iv.(a)} and \textbf{(b)} of Lemma \ref{Lem:1} are equivalent to 
\[
c_{p, j+1} - c_{p-1, j+1} \geq c_{p, j} - c_{p-1, j} \quad {\normalsize \text{   and   }} \quad c_{i+1, p} - c_{i+1, p-1} \geq c_{i, p} - c_{i,p-1}.
\]
Hence, from {\bf iv.(b)} of Lemma \ref{Lem:1} it results the following chain of inequalities.
\[
c_{p-1, q} - c_{p-1, q-1} \geq \ldots \geq c_{i+2, q} - c_{i+2, q-1} \geq c_{i, q} - c_{i, q-1} \geq \ldots \geq c_{1, q} -c_{1, q-1}.
\]
Now, combining the last relationships with (v1) and (v3b), one obtains that for every $i \in [p-1], j \in [q-1]$
\begin{equation*}
\begin{split}
\tfrac{1}{q} &\geq^{{\text{(v1)}}} c_{p-1, q} - c_{p-1,q-1},	\\
		 & \geq^{{\text{iv.(b)}}} c_{i, q} - c_{i, q-1},		\\
		 & \geq^{{\text{(v3b)}}} c_{i, j+1} - c_{ij},		\\
		 & \geq c_{ij} - c_{i, j-1},					\\
		 &\geq^{{\text{(v1)}}} 0, 
\end{split}
\end{equation*}
which proves (q2b). 
Conditions (q2a) can be derived analogously.  
Therefore, $C \in \CDQ_{p,q}$. 
\end{proof}

%---SUBSECTION: Proof of Theorem~\ref{Th:}----
\subsection{Proof of Theorem~\ref{Th:h-repCDQC}.}
By Lemma~\ref{Lem:DQua}, we know that the minimal $H$-representation of $\TCDQ_{p,q}$ is contained within the collection of inequalities (a1),(a3a), and (a3b).  
We here show that the inequalities in the list (a1),(a3a), and (a3b) are exactly the minimal H-representation of $\TCDQ_{p,q}$. 
In particular, we identify $(p \times q)$-matrices $M^{(ij)}_{pq}=[b_{ij}]$, and $H^{(ij)}_{pq}=[h_{ij}]$ for $i \in [p]$ and $j \in [q]$ such that
\begin{itemize}
\item[]{\textsc{\bf Case }(a1).} for every $p$ and $q$, $M^{(11)}_{pq}$ satisfies all inequalities in the list (a1), (a3a), and (a3b) except for inequality of the type $b_{11}\geq 0$.
\item[]{\textsc{\bf Case }(a3a).} for every  $i \in [p-1]$ and $1 \leq j \leq \floor*{\tfrac{q+1}{2}}$, $H^{(ij)}_{pq}$ satisfies all inequalities in the list (a1), (a3a), and (a3b) except for one inequality of the type $ \sum\limits_{h=1}^j b_{i+1, h} \geq \sum\limits_{h=1}^j b_{ih}$. 
\end{itemize}
As shown in Theorem~\ref{Th:h-repUC}'s proof, the matrices we shall identify suffice to prove the thesis as the other inequalities of (a1), (a3a), and (a3b) can be obtained from $M^{(11)}_{pq}$ and $H^{(ij)}_{pq}$ via suitable transformations.

To obtain the thesis it is sufficient to notice that the polytope $\TCDQ_{pq}$ contains $\TUDC_{pq}$. Thus, the matrices {\bf A} and {\bf C1} to {\bf C9} of Theorem~\ref{Th:h-repUC}'s proof are of the type $M^{(11)}_{pq}$ and $H^{(ij)}_{pq}$ for every  $i \in [p-1]$ and $1 \leq j \leq \floor*{\tfrac{q+1}{2}}$. Hence the inequalities (a1),(a3a), and (a3b) are all needed to bound $\TCDQ_{pq}$.
\hfill \qed

%---SUBSECTION: Proof of Theorem~\ref{Th:extQC}----
\subsection{Proof of Theorem~\ref{Th:extQC}.}
Lemma\ref{Lem:DQua} shows each $C \in \CDQ_{p,q}$ to be a discrete quasi-copula.
According to \cite[Theorem 2.3]{quesada_molina_discrete_2005} the checkerboard extension $\tilde{C}$ of $C$ defined as
\[
\tilde{C}(u,v) = (1 - \lambda_u) (1 - \mu_v) c_{ij} + (1 - \lambda_{}u) \mu_v c_{i,j+1} + \lambda_u (1 - \mu_v) c_{i+1, j} + \lambda_u \mu_v c_{i+1, j+1}
\]
where $\tfrac{i}{p} \leq u \leq \tfrac{i + 1}{p}$, $\tfrac{j}{q} \leq v \leq \tfrac{j + 1}{q}$, and
\[ 
\lambda_u = \left\{ \begin{array}{lc} (u - \tfrac{i}{p})p & u > \tfrac{i}{p} \\
1 & u = \tfrac{i}{p} 
\end{array} \right. 
\text{  and  }
 \mu_v = \left\{ \begin{array}{lc} (v - \tfrac{j}{q})q & v > \tfrac{j}{q} \\
1 & v = \tfrac{j}{q} 
\end{array}\right.
\]
is a quasi-copula on $[0,1]^2$ whose restriction on $I_p \times I_q$ is $C$.
Following the same arguments of the proof of Theorem~\ref{Th:ext}, one can check that any arbitrary horizontal section $C_{a}: u \mapsto \tilde{C}(u,a)$, with $a \in [0,1]$, is a convex function. This also works analogously for any arbitrary vertical section.
\hfill $\square$

%---SECTION: Appendix III: Proof from Section~\ref{sec: on vertex representations}----
\section*{Appendix C: Proofs for Section~\ref{sec: on vertex representations}}
\label{sec: appendix C}

\begin{lemma}
	\label{lem: invertible transformation}
	Let $\tau: \R^{p\times q}\longrightarrow \R^{p\times q}$ denote the linear map taking the standard basis vectors $\{ e_{ij} : i\in[p],j\in[q]\}$ to 
	\[
	\tau(e_{ij}) := 
	\begin{cases}
	\sum_{\ell =1}^ie_{kj}-\sum_{k=1}^ie_{k,j+1} 	&	\mbox{ for $i\in[p-1], j\in[q-1]$},	\\
	\sum_{k=1}^qe_{ik}-\sum_{k=1}^qe_{i+1,k} 	&	\mbox{ for $i\in[p-1], j=q$},	\\
	\sum_{k=1}^je_{p-1,k}-\sum_{k=1}^je_{pk} 	&	\mbox{ for $i=p, j\in[q-1]$},	\\
	e_{pq}	&	\mbox{ for $i=p, j=q$}.	\\
	\end{cases}
	\]
	Then $\tau$ is an invertible map with determinant $(-1)^{q-1}q^{p-2}$.  
\end{lemma}

\begin{proof}
	To prove this lemma we will use the matrix representation of $\tau$ and observe that it has a desirable block form from which we can deduce the claimed statements.  
	Let $M = [m_{ij}]$ denote that matrix representation of $\tau$ with respect to the standard basis vectors ordered lexicographically from smallest-to largest along both the rows and columns.  
	(Recall that the lexicographic ordering states that $e_{ij}<e_{st}$ if and only if $i<s$ or $i=s$ and $j<t$.)
	Define the matrices $D_n$, $F_n$, $R_n$, $K_n$ and $L_n$ by
	\def\arraystretch{1.4}
\begin{align*}
d_{ij} &:= 
	\begin{cases}
	1	&	\mbox{if $i=j$ or $j=n$},	\\
	-1 	& 	\mbox{if $i+1=j$},		\\
	0	&	\mbox{otherwise},		\\
	\end{cases} 
& r_{ij} &:= 
	\begin{cases}
	1	&	\mbox{if $i=j$ and $j<n$},	\\
	-1 	& 	\mbox{if $i+1=j$},		\\
	0	&	\mbox{otherwise},		\\
	\end{cases} \\
k_{ij} &:= 
	\begin{cases}
	1	&	\mbox{if $i\leq j$ and $j<n$},	\\
	0	&	\mbox{otherwise},		\\
	\end{cases}
&	f_{ij} &:= 
	\begin{cases}
	1	&	\mbox{if $i=j=n$},	\\
	-1 	& 	\mbox{if $i\leq j$ and $j<n$},		\\
	0	&	\mbox{otherwise},		\\
	\end{cases}\\
\ell_{ij} &:= 
	\begin{cases}
	-1 	& 	\mbox{if $j=n$},		\\
	0	&	\mbox{otherwise},		\\
	\end{cases}
\end{align*}
	Let $M_k$ denote the block matrix consisting of rows $kq+1,\ldots,(k+1)q$ for $k \in \langle p-1 \rangle$.  
	Then, for $k\in[p-3]$
	\[
	M_k = \begin{pmatrix}
	\smash[b]{\block{{\bf 0}_q}{k-1 \mbox{ \scriptsize  times}}}	&	L_q	&	D_q	&	\smash[b]{\block{R_q}{p-k-2 \mbox{ \scriptsize  times}}}		&{\bf 0}_q	\\
	\end{pmatrix},
	\]
	\smallskip
	
   \noindent and for $k \in \{0,p-2,p-1\}$
	\begin{equation*}
	\begin{split}
	M_0 & = 
	\begin{pmatrix}
	D_q	&	R_q	&	\cdots	&	R_q		&{\bf 0}_q	\\
	\end{pmatrix},	\\
	M_{p-2} &=
	\begin{pmatrix}
	{\bf 0}_q	&	{\bf 0}_q	& \cdots	& {\bf 0}_q 	&	L_q	&	D_q	&	K_q	\\
	\end{pmatrix},	\\
	M_{p-1} &=
	\begin{pmatrix}
	{\bf 0}_q	&	{\bf 0}_q	& \cdots	& {\bf 0}_q 	&	L_q	&	F_q		\\
	\end{pmatrix}.	\\
	\end{split}
	\end{equation*}
	Now define the matrices $\widetilde{D}_n$, $\widetilde{R}_n$, $\widetilde{K}_n$, and $\widetilde{F}_n$ where
	\begin{align*}
	\widetilde{d}_{ij} &:= 
	\begin{cases}
	1	&	\mbox{if $i=j$ and $j<n$},	\\
	i 	& 	\mbox{if $j=n$},		\\
	0	&	\mbox{otherwise},		\\
	\end{cases}
&	\widetilde{r}_{ij} &:= 
	\begin{cases}
	1	&	\mbox{if $i=j$ and $j<n$},	\\
	0	&	\mbox{otherwise},		\\
	\end{cases} \\
	\widetilde{k}_{ij} &:= 
	\begin{cases}
	j	&	\mbox{if $i\geq j$ and $j<n$},	\\
	i	&	\mbox{if $i< j$ and $j<n$},	\\
	0	&	\mbox{otherwise},		\\
	\end{cases}
&	\widetilde{f}_{ij} &:= 
	\begin{cases}
	-1+\frac{j}{q}	&	\mbox{if $i\geq j$ and $j<n$},	\\
	\frac{j}{q} 	& 	\mbox{if $i<j$ and $j<n$},		\\
	1 	& 	\mbox{if $i=j=n$},		\\
	0	&	\mbox{otherwise}.	\\
	\end{cases}
	\end{align*}
	Via row reduction, we can reduce $M_0$ to 
	\[
	\widetilde{M}_0 :=
	\begin{pmatrix}
	\widetilde{D}_q	&	\widetilde{R}_q	&	\cdots	&	\widetilde{R}_q		&{\bf 0}_q	\\
	\end{pmatrix},
	\]
	and so we can use the final row of $\widetilde{M}_0$ to reduce $M_1$ to 
	\[
	\widetilde{M}_1 :=
	\begin{pmatrix}
	{\bf 0}_q 	&	\widetilde{D}_q	&	\widetilde{R}_q	&	\cdots	&	\widetilde{R}_q		&{\bf 0}_q	\\
	\end{pmatrix}.
	\]
	It then follows by induction that for $k = 2,\ldots,p-3$ the matrix $M_k$ is reducible to 
	\[
	\widetilde{M}_k :=
	\begin{pmatrix}
	{\bf 0}_q 	&	\cdots	&	{\bf 0}_q	&	\widetilde{D}_q	&	\widetilde{R}_q	&	\cdots	&	\widetilde{R}_q		&{\bf 0}_q	\\
	\end{pmatrix},
	\]
	and that the blocks $M_{p-2}$ and $M_{p-1}$, respectively, are reducible to 
	\begin{equation*}
	\begin{split}
	\widetilde{M}_{p-2} &:=
	\begin{pmatrix}
	{\bf 0}_q 	&	\cdots	&	{\bf 0}_q	&	\widetilde{D}_q	&	\widetilde{K}_q	\\
	\end{pmatrix}, \mbox{ and}	\\
	\widetilde{M}_{p-1} &:=
	\begin{pmatrix}
	{\bf 0}_q 	&	\cdots	&	{\bf 0}_q	&	\widetilde{F}_q		\\
	\end{pmatrix}.		\\
	\end{split}
	\end{equation*}
	From here, basic row reductions can be applied to reduce the block $\widetilde{F}_q$ to the upper triangular matrix
	\[
	\begin{pmatrix}
	-1		&	-1		&	\cdots	&	-1	&	-1		&	-1		\\
	0		&	-1		&			&	-1	&	-1		&	-1		\\
	\vdots	&	\ddots	&	\ddots	&		&	\vdots	&	\vdots	\\
	&			&	\ddots	& \ddots	&	\vdots	&	\vdots	\\
	0		&	0		&	\cdots	&	0	&	-1		&	-1		\\
	0		&	0		&	\cdots	&	0	&	0		&	\frac{1}{q}	\\
	\end{pmatrix}.
	\]
	We may then compute the determinant of $M$ to be
	\[
	\det(M) = \det(\widetilde{F}_q)\cdot\prod_{k=0}^{p-2}\det(\widetilde{D}_q) = (-1)^{q-1}q^{p-2}.
	\]
	In particular, since $\det(M)\neq 0$, we conclude that $\tau$ is invertible.
\end{proof}

%---SUBSECTION: Proof of Theorem~\ref{thm: recursively constructed vertices}-----
\subsection{Proof of Theorem~\ref{thm: recursively constructed vertices}.}

	In the following, we work with the polytope $\TUDC_{p,q}$.  
	However, the same argument works for $\TCDQ_{p,q}$.  
	So as to apply Theorem~\ref{thm: vertices}, we must show that $(p+s)(q+t)$ facet-defining inequalities of $\TUDC(({\bf q}_p,{\bf t}_s),({\bf s}_t,{\bf p}_q))$ with linearly independent facet-normals are active on $B\oplus D$.  
	Notice that since $B$ and $D$ are submatrices of $B\oplus D$, then the $pq$ and $st$ (respectively) inequalities that are active on each of $B$ and $D$ and have linearly independent normal vectors all yield active inequalities on $B\oplus D$ that have linearly independent normal vectors.  
	There are also $pt$ inequalities of the form 
	\[
	\sum_{k=1}^ie_{k j} - \sum_{k=1}^ie_{kj-1}\geq 0
	\]
	that are active on the submatrix ${\bf 0}_{p,t}$ of $B\oplus D$, and there are $sq$ inequalities contained in the list 
	\begin{enumerate}
		\item[(1)] $\sum_{k=1}^ie_{ki}-\sum_{k=1}^ie_{k,j+1}\leq 0$ for $i\in[p+s-1], j\in[q+t-1]$,
		\item[(2)] $\sum_{k=1}^qe_{ik}-\sum_{k=1}^qe_{i+1,k}\leq 0$ for $i\in[p+s-1]$,
		\item[(3)] $\sum_{k=1}^je_{p+s-1,k}-\sum_{k=1}^je_{p+s,k}\leq 0$ for $j\in[q+t-1]$,	
		\item[(4)] $e_{pq}\geq 0$,
	\end{enumerate}
	that are active on the submatrix ${\bf 0}_{s,q}$ of $B\oplus D$.  
	These can be seen to have linearly independent facet-normals from those given by the submatrices $B$ and $D$.  
	Moreover, by Lemma~\ref{lem: invertible transformation}, all such inequalities have linearly independent facet-normals from one another.   
	Thus, we conclude that $B\oplus D$ is a vertex of $\TUDC(({\bf q}_p,{\bf t}_s),({\bf s}_t,{\bf p}_q))$.   \hfill $\square$

%---SUBSECTION: Proof of Proposition~\ref{prop: generating functions}-----
\subsection{Proof of Corollary~\ref{prop: generating functions}.}
To prove this corollary, we first recall that a \emph{(weak) composition} of a positive integer $p\in\Z_{>0}$ with $k$ parts is a sum 
$
c_1+c_2+\cdots+c_k = p,
$
in which the order of the summands $c_1,\ldots,c_k\in \Z_{>0}$ matters.  
It follows that if $C\in\R^{p\times p}$ is a decomposable vertex of $\UDC_p$, then there exists a composition $c_1+c_2+\cdots+c_k = p$ such that there are indecomposable matrices $C_1\in\ID_{c_1},\ldots,C_k\in\ID_{c_k}$ such that 
\[
 C = C_1 \oplus \cdots \oplus C_k.
 \]
 It then follows that
 \begin{equation*}
 \begin{split}
 \D(x) &= \sum_{k\geq0}\left(\sum_{\ell\geq0}|\ID_\ell|x^\ell\right)^k,	\\
  	&= \sum_{k\geq0}\left(\ID(x)\right)^k,	\\
	&= \sum_{k\geq0}\left(V(x) - \D(x) \right)^k,	\\
	&= \frac{1}{1+\D(x)-V(x)}.	\\
\end{split}
\end{equation*}
From this it is quick to conclude that
\[
V(x) = \frac{\D(x)^2+\D(x)-1}{\D(x)}.
\]
In a similar fashion, the inequality follows.
\hfill $\square$

%---SECTION: Appendix IV: Proofs from Section~\ref{sec: aggregation functions and alternating transportation polytopes}----
\section*{Appendix D: Proofs for Section~\ref{sec: aggregation functions and alternating transportation polytopes}}
\label{sec: appendix D}

%---SUBSECTION: Proof of Proposition~\ref{Prop:AggrFunTranPol}}.
\subsection{Proof of Proposition~\ref{Prop:AggrFuncTranPol}.}
(i) $\Rightarrow$ (ii)
We consider $C_{pq} \in \text{SAF}(\tilde{u},\tilde{v})$. For every  $i \in \langle p \rangle$, $j \in \langle q \rangle$, we can take $C_{pq}$ to a $(p \times q)$ matrix $[x_{ij}]$ through the following linear transformation
\[x_{ij}= pq (c_{ij} + c_{i-1,j-1} - c_{i-1,j} - c_{i,j-1}).\]
We here show that the new constructed matrix $[x_{ij}]$ lies in the transportation polytope $\mathcal{T}(u,v)$ whose margins are the vectors $u \in \mathbb{R}^p$ and $v \in \mathbb{R}^q$, such that for every $i \in [p]$, $u_i:= \tilde{u}_i-\tilde{u}_{i-1}$, and $j \in[q]$, $v_j:= \tilde{v}_j-\tilde{v}_{j-1}$.
Indeed, condition (AF2) implies that $x_{ij} \geq 0$ for every $i \in \langle p \rangle$, $j \in \langle q \rangle$.
By construction, one has that $\sum_{h=1}^q x_{ih} =pq (c_{iq} - c_{i-1,q}) = \tilde{u}_i - \tilde{u}_{i-1} = u_i$.
Similarly, it follows that $\sum_{\ell=1}^p x_{\ell j}= v_j$. Hence, the thesis.

(i) $\Leftarrow$ (ii)
We here verify that every $C_{pq}$ defined as in equation~(\ref{Eq:SAF_TP}) belongs to the set $\text{SAF}(\tilde{u},\tilde{v})$, with margins given by the vectors $\tilde{u} \in \mathbb{R}^p$ and $\tilde{v} \in \mathbb{R}^q$, whose values are defined for every $i \in [p]$, as $\tilde{u}_i:= \sum_{\ell=1}^{i}{u}_{\ell}$, and for $j \in[q]$, as $\tilde{v}_j:= \sum_{h=1}^j v_h$. 
Clearly, any such matrix $C_{pq}$ satisfies condition (AF2). 
Since the empty sum equals zero by convention, (AF1a) holds as well.
It remains to show the validity of (AF1b). From equation~(\ref{Eq:SAF_TP}), one has that 
$$c_{pj}=\frac{1}{pq} \sum_{\ell =1}^p \sum_{h=1}^j x_{\ell h} = \frac{1}{pq} \sum_{h=1}^j v_h < \frac{1}{pq} \sum_{h=1}^{j+1} v_h = c_{p,j+1}$$
A similar argument applied to the columns completes the proof. 
\hfill $\square$
%

%---SUBSECTION: Proof of Proposition~\ref{Prop:AggrFunAltPol}}.
\subsection{Proof of Proposition~\ref{Prop:AggrFunAltPol}.}
(i) $\Rightarrow$ (ii)
We consider $C_{pq} \in \text{ASA}(\tilde{u},\tilde{v})$. 
For every  $i \in \langle p \rangle$, $j \in \langle q \rangle$, we can take $C_{pq}$ to a $(p \times q)$ matrix $[x_{ij}]$ through the following linear transformation
\[x_{ij}= pq (c_{ij} + c_{i-1,j-1} - c_{i-1,j} - c_{i,j-1}).\]
The new constructed matrix $[x_{ij}]$ lies in the alternating transportation polytope $\mathcal{A}(u,v)$ whose margins are the vectors $u \in \mathbb{R}^p$ and $v \in \mathbb{R}^q$, such that for every $i \in [p]$, $u_i:= \tilde{u}_i-\tilde{u}_{i-1}$, and $j \in[q]$, $v_j:= \tilde{v}_j-\tilde{v}_{j-1}$.
According to Proposition~\ref{Prop:AggrFuncTranPol}'s proof, one can derive the marginal constraints of $[x_{ij}]$ from (AF1a) and (AF1b).
It remains to verify that $0 \leq \sum_{\ell=1}^i x_{\ell j} \leq v_j$, and $0 \leq \sum_{h=1}^j x_{i h} \leq u_i$, for every $i \in \langle p \rangle$, $j \in \langle q \rangle$. 
It is useful to observe that $\sum_{\ell=1}^i x_{\ell j} = pq \sum_{\ell=1}^i (c_{\ell j} + c_{\ell-1,j-1} - c_{\ell-1,j} - c_{\ell,j-1})= pq (c_{i j} - c_{i,j-1})$.
We now notice that for every $i \in \langle p \rangle$ and $j \in \langle q \rangle$, one has
$(c_{ij} - c_{i,j-1} - c_{0 j} + c_{0, j-1}) \geq 0$,
from (AF2b) and (AF1a). Hence $\sum_{\ell=1}^i x_{\ell j} \geq 0$. 
Moreover, from (AF2b) and (AF1b), one has 
$(c_{i,j-1} - c_{ij} - c_{p,j-1} + c_{pj}) \geq 0$.
Thus, $c_{ij} - c_{i,j-1} \leq c_{pj} - c_{p,j-1}$ and $\sum_{\ell=1}^i x_{\ell j} \leq \tilde{v}_j - \tilde{v}_{j-1}= v_j$.
The remaining conditions on the row sums can be derived in a similar fashion.

(i) $\Leftarrow$ (ii)
We now prove that every $C_{pq}$ defined as in equation~(\ref{Eq:ASA_AT}) belongs to the set $\text{ASA}(\tilde{u},\tilde{v})$, with margins given by the vectors $\tilde{u} \in \mathbb{R}^p$ and $\tilde{v} \in \mathbb{R}^q$, whose values are defined for every $i \in [p]$, as $\tilde{u}_i:= \sum_{\ell=1}^{i}{u}_{\ell}$, and for $j \in[q]$, as $\tilde{v}_j:= \sum_{h=1}^j v_h$. 
Conditions (AF1a) and (AF1b) can be derived according to Proposition~\ref{Prop:AggrFuncTranPol}'s proof.
We notice that $c_{i_1\,j_1} + c_{i_2\,j_2} - c_{i_1\,j_2} - c_{i_2\,j_1}$ can be expressed as
\[
\sum_{\ell=1}^{i_1} \sum_{h=1}^{j_1} x_{\ell h} + \sum_{\ell=1}^{i_2} \sum_{h=1}^{j_2} x_{\ell h}
 - \sum_{\ell=1}^{i_1} \sum_{h=1}^{j_2} x_{\ell h} - \sum_{\ell=1}^{i_2} \sum_{h=1}^{j_1} x_{\ell h}.
\]
Hence, the above formulation becomes $\sum_{\ell=1}^{i_2} \left(x_{\ell,j_1+1}+\ldots+x_{\ell\,j_2}\right)$, when $i_1=0$, and $\sum_{\ell=i_1+1}^{p} \left(x_{\ell,j_1+1}+\ldots+x_{\ell\,j_2}\right)$, if $i_2=p$. 
In either case, the sums are nonnegative. In similar way, one can derive the cases $j_2=q$ and $j_1=0$.
\hfill $\square$

\end{document}